\documentclass[a4paper]{amsart}
\usepackage{graphicx}
\usepackage{amsmath}
\usepackage{amssymb}
\usepackage{amsfonts}
\usepackage{amsthm}
\usepackage[all]{xy}
\usepackage{eucal}

\newcommand{\A}{\mathbf{A}}
\newcommand{\NA}{N\!\mathbf{A}}
\newcommand{\sSets}{\mathbf{sSets}}

\theoremstyle{plain}
\newtheorem{theorem}{Theorem}[section]
\newtheorem{theorema}{Theorem}

\newtheorem{lemma}[theorem]{Lemma}
\newtheorem{proposition}[theorem]{Proposition}
\newtheorem{propositiona}[theorema]{Proposition}

\newtheorem{corollary}[theorem]{Corollary}
\newtheorem{corollarya}[theorema]{Corollary}

\newtheorem*{proposition*}{Proposition}
\newtheorem*{corollary*}{Corollary}

\theoremstyle{definition}

\newtheorem{remark}[theorem]{Remark}

\makeatletter
\@addtoreset{theorem}{section}
\@addtoreset{subsection}{section}
\makeatother

\DeclareMathAlphabet{\mathpzc}{OT1}{pzc}{m}{it}

\begin{document}

\title{Left fibrations and homotopy colimits}

\author[Gijs Heuts]{Gijs Heuts}
\address{Harvard University, Department of Mathematics, 1 Oxford Street, 02138 Cambridge, Massachusetts, USA}
\email{gheuts@math.harvard.edu}

\author[Ieke Moerdijk]{Ieke Moerdijk}
\address{Radboud Universiteit Nijmegen, Institute for Mathematics, Astrophysics and Particle Physics, Heyendaalseweg 135, 6525 AJ Nijmegen, The Netherlands}
\email{i.moerdijk@math.ru.nl}

\date{}

\begin{abstract}
For a small category $\A$, we prove that the homotopy colimit functor from the category of simplicial diagrams on $\A$ to the category of simplicial sets over the nerve of $\A$ establishes a left Quillen equivalence between the projective (or Reedy) model structure on the former category and the covariant model structure on the latter. We compare this equivalence to a Quillen equivalence in the opposite direction previously established by Lurie. From our results we deduce that a categorical equivalence of simplicial sets induces a Quillen equivalence on the corresponding over-categories, equipped with the covariant model structures. Also, we show that versions of Quillen's Theorems A and B for $\infty$-categories easily follow.
\end{abstract}

\maketitle


\section{Introduction and main results}

Let $\A$ be a small category and consider the category $\sSets^{\A}$ of functors from $\A$ to the category of simplicial sets. Such functors are also frequently referred to as \emph{simplicial diagrams} on $\A$, or \emph{simplicial presheaves} on the opposite category $\A^{\mathrm{op}}$. The well-known construction of \emph{homotopy colimits} provides a functor
\begin{equation*}
h_!: \sSets^{\A} \longrightarrow \sSets / \NA 
\end{equation*}
from the category of simplicial diagrams on $\A$ to the category of simplicial sets over the nerve of $\A$. It maps a diagram $F$ to the simplicial set $h_!(F)$ with as $n$-simplices the pairs $(x, a_0 \rightarrow \cdots \rightarrow a_n)$ consisting of an $n$-simplex $a_0 \rightarrow \cdots \rightarrow a_n$ in $\NA$ and an $n$-simplex $x$ in $F(a_0)$. There is also a functor
\begin{equation*}
r_!: \sSets / \NA \longrightarrow \sSets^{\A}
\end{equation*}
in the other direction, which we will refer to as the \emph{rectification functor}. For a simplicial set over $\NA$, say $\pi: X \longrightarrow \NA$, we abusively denote its image under $r_!$ by $r_!(X)$, leaving the map $\pi$ implicit. Its value $r_!(X)(b)$ at an object $b$ of $\A$ has as $n$-simplices the pairs $(x, \beta)$ with $x \in X_n$ and $\beta: \pi(x_n) \rightarrow b$ a morphism in $\A$. (Here $x_n$ denotes the last vertex of $x$.) Equivalently, it may be described by
\begin{equation*}
r_!(X)(b) = X \times_{\NA} N(\A/b).
\end{equation*}
This functor is closely related to the construction of a split cofibered category out of an arbitrary category over $\A$ \cite{sga1}. \par
The goal of this paper is to analyze the homotopy colimit functor $h_!$ from the point of view of model categories. To this end, we consider the projective (and possibly Reedy) model structures on $\sSets^{\A}$ and the covariant model structure on $\sSets/\NA$. (The definitions of these are reviewed in Section \ref{sec:modelstructures}.) We will give direct, complete and self-contained proofs of the following three statements: 

\begin{propositiona}
\label{prop:A}
\begin{itemize}
\item[(a)] The homotopy colimit functor $h_!$ is part of a Quillen pair (left adjoint on the left)
\[
\xymatrix{
h_!: \sSets^{\A} \ar@<.5ex>[r] & \sSets/\NA : h^* \ar@<.5ex>[l]
}
\]
with respect to the projective model structure on $\sSets^\A$ and the covariant model structure on $\sSets/\NA$.
\item[(b)] In case $\A$ is a (generalized) Reedy category, the same is true for the Reedy model structure on $\sSets^\A$.
\end{itemize}
\end{propositiona}   

\begin{propositiona}
\label{prop:B}
The rectification functor $r_!$ is part of a Quillen pair
\[
\xymatrix{
r_!: \sSets/\NA \ar@<.5ex>[r] & \sSets^\A : r^* \ar@<.5ex>[l]
}
\]
between the covariant and projective model structures.
\end{propositiona}

Note that since the identity functor on $\sSets^\A$ is a left Quillen equivalence from the projective to the Reedy model structure (if $\A$ is indeed generalized Reedy), part (b) provides additional information in Proposition \ref{prop:A}, but the analogous statement for Proposition \ref{prop:B} just follows by composing the two Quillen pairs. The same remark applies to the following theorem:

\begin{theorema}
\label{thm:c}
The two Quillen pairs (left adjoints on top)
\[
\xymatrix{
\sSets^\A \ar@<.5ex>[r]^-{h_!} & \sSets/\NA \ar@<.5ex>[l]^-{h^*}\ar@<.5ex>[r]^-{r_!} & \sSets^\A \ar@<.5ex>[l]^-{r^*}
}
\]
are Quillen equivalences between the covariant and projective model structures. Furthermore, the left derived functor of $h_!$ and the right derived functor of $r^*$ are naturally equivalent.
\end{theorema}

We will prove this theorem by exhibiting, for cofibrant simplicial diagrams $F$, a natural weak equivalence $r_!h_!(F) \longrightarrow F$
and for objects $X \in \sSets/\NA$ a natural zigzag of weak equivalences between $h_!r_!(X)$ and $X$. This then shows that the left derived functors of $h_!$ and $r_!$ are mutually inverse and are therefore both part of Quillen equivalences. It also shows that the derived functors of $h_!$ and $r^*$ define naturally isomorphic functors on the level of homotopy categories;
\begin{equation*}
\mathbf{L}h_! \simeq \mathbf{R}r^*: \mathrm{Ho}(\sSets^\A) \longrightarrow \mathrm{Ho}(\sSets/\NA).
\end{equation*}

Finally, we will deduce several consequences of our results. The following was already proved by Dugger (see Theorem 5.2 of \cite{dugger}) in the special case where $\NA$ is contractible:

\begin{corollarya}
\label{prop:localized}
After localizing the projective model structure on $\sSets^\A$ with respect to all the maps $\A(b,-) \rightarrow \A(a,-)$ induced by morphisms $a \rightarrow b$ in $\A$, the functor $h_!$ induces a Quillen equivalence between this model structure and the Kan-Quillen model structure on the category $\sSets/\NA$. 
\end{corollarya}

To make this a little more concrete, suppose for example that $F$ is a simplicial diagram on $\A$ such that $F(a)$ is a Kan complex for every object $a$ and $F(f)$ is a weak equivalence for every morphism $f$. Such a diagram is a fibrant object of $\sSets^\A$ in the localized model structure of Corollary \ref{prop:localized}. The pair $(r_!,r^*)$ is still a Quillen pair for the localized model structures of this corollary (since $\mathbf{L}r_!$ is homotopy inverse to $\mathbf{L}h_!$), so that $r^*F$ is a Kan fibration over $\NA$.  Note that the fibers of $r^*F$ over vertices of $\NA$ are precisely the corresponding values of $F$. This shows that any diagram of weak equivalences between Kan complexes can be `realized' by a Kan fibration, which is an interesting observation in its own right.

In the special case where $\A$ is a groupoid, the localizing maps of Corollary \ref{prop:localized} are natural isomorphisms; therefore, localizing with respect to them does not change the model structure. Combining this observation with Theorem \ref{thm:c} then yields the corollary below. This is a well-known fact; for example, in case $\A$ is a group $G$, this reproduces the equivalence of \cite{drordwyerkan} between the model categories of simplicial sets with $G$-action and simplicial sets over the classifying space $BG$. See also \cite{hollander,jardine}.

\begin{corollarya}
If $\A$ is a groupoid, the functor $h_!$ induces a Quillen equivalence between the projective model structure on $\sSets^\A$ and the Kan-Quillen model structure on $\sSets/\NA$.
\end{corollarya}

We will use our results to demonstrate the following invariance property of the covariant model structure:
 
\begin{propositiona}
\label{prop:invariance}
If $f: X \longrightarrow Y$ is a categorical equivalence of simplicial sets (i.e. a weak equivalence in the Joyal model structure), then the adjoint pair
\[
\xymatrix{
\mathbf{sSets}/X \ar@<.5ex>[r]^{f_!} & \mathbf{sSets}/Y, \ar@<.5ex>[l]^{f^*}
}
\]
given by composing with and pulling back along $f$, is a Quillen equivalence.
\end{propositiona}

To conclude, we will show how versions of Quillen's Theorems A and B for $\infty$-categories can be proved using our methods. Recall that for a simplicial set $B$ and $b$ a vertex of $B$, one can define a simplicial set $B/b$, which has as its $n$-simplices the $n+1$-simplices of $B$ whose final vertex is $b$, with face and degeneracy maps defined in an obvious manner. Moreover, for a map $X \rightarrow B$ one defines
\begin{equation*}
X/b := X \times_B B/b.
\end{equation*}
The following is a refinement of Quillen's Theorem A:

\begin{propositiona}
\label{prop:QuillenA}
Suppose $B$ is an $\infty$-category and we have a diagram
\[
\xymatrix{
X \ar[rr]^f \ar[dr] & & Y \ar[dl] \\
& B. &
}
\]
If, for every vertex $b$ of $B$, the map $f$ induces a weak equivalence of simplicial sets $X/b \rightarrow Y/b$ (in the Kan-Quillen model structure), then $f$ itself is a weak equivalence in the covariant model structure over $B$. 
\end{propositiona}

If $\beta$ is an $m$-simplex of $B$, one can define a simplicial set $X/\beta$ over $B$ in the same way one defines $X/b$ as above. Indeed, one considers instead the simplicial set $B/\beta$, whose $n$-simplices are $n+m+1$-simplices of $B$ which restrict to $\beta$ on the final segment $\{n+1, \ldots, n+m+1\}$. Our last result generalizes Quillen's Theorem B:

\begin{propositiona}
\label{prop:QuillenB}
Let $p: X \rightarrow B$ be a map of simplicial sets and suppose $B$ is an $\infty$-category. Furthermore, assume that for any 1-simplex $\beta$ of $B$, the map $X/\beta \rightarrow X/d_0\beta$ is a weak equivalence in the Kan-Quillen model structure. Then $X/b$ is a model for the homotopy fiber of $p$ over $b$ in the Kan-Quillen model structure.
\end{propositiona}

Our main Theorem C is only partially new. Indeed, at least in the special case $\A = \mathbf{\Delta}^{\mathrm{op}}$, it is a classical fact that $h_!$ and $r^*$ are two models for the homotopy colimit, and it is proved in detail in \cite{cegarraremedios} that these functors produce objects which are equivalent in the classical Kan-Quillen homotopy theory of simplicial sets. The introduction to \cite{cegarraremedios}, which uses the perhaps more familiar notation $\overline{W}$ for the functor we call $r^*$, provides an overview of occurrences of this functor in the literature and connects it to classical topics in the homotopy theory of simplicial sets and simplicial groups. \par 
Our result that the derived functors of $h_!$ and $r^*$ are naturally equivalent for the covariant model structure over $\NA$ is sharper and gives additional significance to this model structure. The fact that $r_!$ and $r^*$ constitute a Quillen equivalence is proved in a more general context by Lurie \cite{htt}; indeed, the functor we denote by $r^*$ coincides with what Lurie calls the \emph{relative nerve functor} in his Section 3.2.5. However, his proof uses rather a lot of machinery and is essentially different from ours, since it does not exploit the relation between $r^*$ and the standard model $h_!$ for the homotopy colimit. Our Proposition \ref{prop:invariance} already follows from Lurie's work (see his Remark 2.1.4.11). However, the proof we provide is more direct. Also, variants of Proposition \ref{prop:QuillenA} were already proved by Joyal \cite{joyal} and Lurie \cite{htt}, but we provide a different and concise proof, using only the results in this paper. \par 
The virtue of this note, if any, thus lies in the fact that it provides a direct and self-contained proof of the stated facts and additionally yields left Quillen functors between $\mathbf{sSets}^{\A}$ and $\mathbf{sSets}/\NA$ in both directions. 

\section{Review of several model structures}
\label{sec:modelstructures}

In this section we recall the details of the model categories involved in the main results stated above. As before, $\A$ is a fixed small category.

\subsection{The projective model structure on $\sSets^\A$}
\label{subsec:proj}
The material in this section is standard and can for example be found in \cite{hirschhorn}. The projective model structure is determined by specifying that a map $G \rightarrow F$ of simplicial diagrams on $\A$ is a fibration, respectively a weak equivalence, if and only if for each object $a$ in $\A$ the map $G(a) \rightarrow F(a)$ is a fibration, respectively a weak equivalence, with respect to the classical Kan-Quillen model structure on simplicial sets. This model structure is cofibrantly generated, with generating cofibrations and trivial cofibrations of the form
\begin{equation*}
\partial \Delta^n \times \A(a,-) \longrightarrow \Delta^n \times \A(a,-) \quad\quad\quad (n \geq 0)
\end{equation*}
and
\begin{equation*}
\Lambda_k^n \times \A(a,-) \longrightarrow \Delta^n \times \A(a,-) \quad\quad\quad (n \geq 0, 0 \leq k \leq n)
\end{equation*}
respectively. We refer to the three distinguished classes of maps in this model structure as \emph{projective (co)fibrations} and \emph{projective weak equivalences}. The projective model structure is left proper (as well as right proper) and simplicial. For simplicial diagrams $F$ and $G$ on $\A$ and a simplicial set $M$, the simplicial structure is given by
\begin{eqnarray*}
(M \otimes F)(a) & = & M \times F(a) \\
\mathrm{Map}(F,G)_n & = & \mathrm{Hom}(\Delta^n \otimes F, G)
\end{eqnarray*}
where $\mathrm{Hom}$ denotes the set of morphisms in $\sSets^\A$.

\subsection{The Reedy model structure}
\label{subsec:reedy}
If $\A$ is a (generalized) Reedy category \cite{bergermoerdijkReedy}, there is another useful model structure on $\sSets^\A$ with the same weak equivalences as the projective structure. We will write $\A^-$ and $\A^+$ for the subcategories given by the generalized Reedy structure, so that any arrow $b \rightarrow a$ in $\A$ factors as
\[
\xymatrix{
b \ar[r]^{-} &  c \ar[r]^{+} & a
}
\]
where $b \rightarrow c$ is in $\A^-$ and $c \rightarrow a$ is in $\A^+$, in a way which is unique up to unique isomorphism. We refer to \cite{bergermoerdijkReedy} for a complete list of axioms such a structure should satisfy. For an object $b$ of $\A$, let us write $\A^-(b,-)$ for the subfunctor of the representable functor $\A(b,-)$ whose value $\A^-(b,a)$ consists of those morphisms $b \rightarrow a$ which are \emph{not} in $\A^+$, or equivalently, which admit a factorization $b \rightarrow c \rightarrow a$ with $b \rightarrow c$ in $\A^-$ and \emph{not} an isomorphism. Then a map of simplicial diagrams $G \rightarrow F$ is a Reedy fibration if and only if for each $b \in \A$, it has the right lifting property with respect to the maps
\begin{equation*}
\Lambda_k^n \times \A(b,-) \cup \Delta^n \times \A^-(b,-) \longrightarrow \Delta^n \times \A(b,-) \quad\quad\quad (n \geq 0, \, 0 \leq k \leq n)
\end{equation*}
Indeed, it easily follows from the definition of a generalized Reedy category that the map
\begin{equation*}
\varinjlim_{b \rightarrow c} \A(c,-) \rightarrow \A^-(b,-)
\end{equation*}
is an isomorphism, where the colimit is over all non-isomorphisms $b \rightarrow c$ in $\A^-$. \par 
This Reedy model structure is again cofibrantly generated, left proper (as well as right proper) and simplicial. Moreover, the identity functor is a left Quillen equivalence from the projective to the Reedy model structure.

\subsection{The covariant model structure on $\sSets/\NA$}
This model structure is treated by Joyal \cite{joyal} and Lurie \cite{htt}. Let us first recall the Joyal model structure on simplicial sets; it is uniquely determined by stating that its cofibrations are the monomorphisms, while its fibrant objects are the $\infty$-categories (or quasicategories), i.e. simplicial sets having the extension property with respect to the \emph{inner} horns $\Lambda_k^n \rightarrow \Delta^n$ (for $n > 1$ and $0 < k < n$). (In particular, such inner horn inclusions are trivial cofibrations in the Joyal model structure.) The fibrations between fibrant objects are the maps having the right lifting property with respect to these inner horns, as well as with respect to the inclusion $\{0\} \rightarrow J$, where $J$ denotes the nerve of the contractible groupoid 
\[
\xymatrix{
0 \ar@{<->}[r]^\simeq & 1.
}
\]
The Joyal model structure is Cartesian, but not simplicial. Any simplicial set $B$ defines an induced Joyal model structure on the slice category $\sSets/B$. The \emph{covariant model structure} on $\sSets/B$ is the left Bousfield localization of this Joyal structure along the \emph{left horn inclusions} over $B$, i.e. maps of the form
\[
\xymatrix{
\Lambda_0^n \ar[rr]\ar[dr] & & \Delta^n \ar[dl] \\
& B & 
}
\]
for $n \geq 1$. We refer to the distinguished classes of maps in the covariant model structure as \emph{covariant (co)fibrations} and \emph{covariant weak equivalences}. The fibrant objects are the \emph{left fibrations} over $B$, i.e. those maps $X \rightarrow B$ having the right lifting property with respect to the maps $\Lambda_k^n \rightarrow \Delta^n$ for $n\geq 1$ and $0 \leq k < n$. The saturation of this set of horn inclusions is called the class of \emph{left anodyne maps}. In particular, left anodyne maps are covariant trivial cofibrations. Having the right lifting property with respect to these maps also characterizes the fibrations between fibrant objects in the covariant model structure. The weak equivalences between fibrant objects are precisely those maps that induce homotopy equivalences on fibers \cite{joyal}. \par 
The covariant model structure is left proper and simplicial (see Theorem 11.18 of \cite{joyal} or Proposition 2.1.4.8 of \cite{htt}). For a simplicial set $M$ and maps $X \rightarrow B$ and $Y \rightarrow B$, the objects
\begin{equation*}
M \otimes (X \rightarrow B) \quad\quad \text{and} \quad\quad \mathrm{Map}(X \rightarrow B, Y \rightarrow B)
\end{equation*}
are the composition $M \times X \rightarrow X \rightarrow B$ and the simplicial set with $n$-simplices the maps $\Delta^n \times X \rightarrow Y$ over $B$. We will simply write $\mathrm{Map}_B(X,Y)$ for this simplicial set. Since the covariant model structure is simplicial, a weak equivalence $M \rightarrow N$ of simplicial sets (in the Kan-Quillen model structure) gives a covariant weak equivalence $M \times X \rightarrow N \times X$ over $B$. \par 
The following basic `left homotopy' lemma will be very useful in this paper:

\begin{lemma}
\label{lem:lefthomotopy}
Consider a monomorphism $i$ of simplicial sets over $B$:
\[
\xymatrix{
X \ar[rr]^i \ar[dr] && Y \ar[dl] \\
& B. &
}
\]
Now suppose there exist a retraction $r: Y \rightarrow X$ so that $ri = \mathrm{id}_X$ and a homotopy $h: \Delta^1 \times Y \rightarrow Y$ (relative to $X$) from $ir$ to $\mathrm{id}_Y$, i.e. satisfying
\begin{eqnarray*}
h_0 = h |_{\{0\} \times Y} = ir: Y \longrightarrow Y, && \\
h_1 = h |_{\{1\} \times Y} = \mathrm{id}_Y: Y \longrightarrow Y. &&
\end{eqnarray*}
Then $i$ is a trivial cofibration in the covariant model structure over $B$. \end{lemma}
\begin{proof}
Consider a diagram
\[
\xymatrix{
X \ar[d]_i \ar[r]^f & W \ar[d]^p \ar[dr] & \\
Y \ar[r]_g \ar@{-->}[ur] & Z \ar[r] & B,
}
\]
in which $p$ is a fibration in the covariant model structure over $B$. We want to show that a lift as indicated by the dashed arrow exists. Now form the square
\[
\xymatrix@C=70pt{
\Delta^1 \times X \cup_{\{0\} \times X} \{0\} \times Y \ar[r]^-{f\pi_X \cup fr}\ar[d] & W \ar[d]^p \\
\Delta^1 \times Y \ar[r]_{gh} \ar@{-->}_k[ur] & Z,
}
\]
where $\pi_X: \Delta^1 \times X \rightarrow X$ is the projection. The left vertical map is the pushout-product of the left anodyne map $\{0\} \rightarrow \Delta^1$ with a cofibration of simplicial sets and hence left anodyne itself, by a result of Joyal \cite{joyal}. Therefore a lift $k$ exists in the diagram. The map $k_1 = k|_{Y \times \{1\}}$ is now a lift in the original diagram. 
\end{proof}

We will refer to a tuple $(i,r,h)$ as in the lemma as a \emph{covariant deformation retract}.

\begin{remark}
Note that the maps $r$ and $h$ of the previous lemma are just maps of simplicial sets, \emph{not} required to be compatible with the maps down to $B$. Also, the direction of the homotopy, from $ir$ to $\mathrm{id}_Y$, is essential.
\end{remark}

\begin{remark}
The previous lemma shows the importance of \emph{left homotopies}, or \emph{directed homotopies}, in the covariant model structure. In fact, more can be said: the covariant model structure is the minimal simplicial localization of the Joyal model structure on $\mathbf{sSets}/B$ which satisfies the statement of Lemma \ref{lem:lefthomotopy}. We will not need this fact and therefore omit its proof, although it is straightforward.
\end{remark}

For a vertex $b$ of a simplicial set $B$, denote by $b/B$ the simplicial set whose $n$-simplices are the $n+1$-simplices of $B$ whose initial vertex is $b$, and whose face and degeneracy maps are defined by considering the $n+1$-simplex as the join $\Delta^0 \star \Delta^n$ and applying the face and degeneracy maps of $\Delta^n$. Said differently, $b/B$ can be thought of as the fiber of 
\begin{equation*}
\lambda: \mathrm{Dec}(B) \longrightarrow^{} B_0
\end{equation*}
over $b$, where $\mathrm{Dec}$ is the d\'{e}calage functor (cf. Illusie \cite{illusie}) and $\lambda$ denotes projection onto the initial vertex. Note that $d_0$ induces a map $\mathrm{Dec}(B) \rightarrow B$, which restricts to a map $b/B \rightarrow B$. We will frequently use the following set of lemmas.

\begin{lemma}
\label{lem:leftanod1}
Let $b$ be an object of $B$. Then the diagram
\[
\xymatrix{
\Delta^0 \ar[rr]^{s_0}\ar[dr]_b & & b/B \ar[dl]^{d_0} \\
& B, &
}
\]
is part of a covariant deformation retract and in particular is a covariant trivial cofibration in $\sSets/B$.
\end{lemma}
\begin{proof}
Denote the unique map $b/B \rightarrow \Delta^0$ by $r$. Also, for an $n$-simplex $(\alpha,\beta)$ of $\Delta^1 \times b/B$, let $k$ denote the minimal integer in $\{0, \ldots, n\}$ such that $\alpha(k) = 1$, if it exists. We may informally represent such a simplex as follows:
\[
\xymatrix@R=5pt{
b \ar[r] & b_0 \ar[r] & \cdots \ar[r] & b_{k-1} \ar[r] & b_k \ar[r] & \cdots \ar[r] & b_n \\
& 0 & \cdots & 0 & 1 & \cdots & 1.
}
\] 
We then define a homotopy $h: \Delta^1 \times b/B \rightarrow b/B$ by mapping $(\alpha,\beta)$ to the simplex represented by
\[
\xymatrix{
b \ar@{=}[r] & b \ar@{=}[r] & \cdots \ar@{=}[r] & b \ar[r] & b_k \ar[r] & \cdots \ar[r] & b_n,
}
\]
where $b \rightarrow b_k \rightarrow \cdots \rightarrow b_n$ denotes the simplex $d_0^k \beta$ of $b/B$ and the equal signs denote the result of taking $k$ iterated degeneracy maps $s_0$ in the simplicial set $B$. Then $r$ and $h$ satisfy the hypotheses of Lemma \ref{lem:lefthomotopy}.
\end{proof}

In similar fashion, one proves the following:

\begin{lemma}
\label{lem:leftanod2}
Let $\beta$ be an $n$-simplex of $B$. Then
\[
\xymatrix{
\Delta^0 \ar[rr]^0\ar[dr]_{\beta_0} & & \Delta^n \ar[dl]^\beta \\
& B &
}
\]
is part of a covariant deformation retract in $\sSets/B$.
\end{lemma}
\begin{proof}
Again, let $r: \Delta^n \rightarrow \Delta^0$ be the obvious map and let $h: \Delta^1 \times \Delta^n \rightarrow \Delta^n$ be the map that is determined uniquely by the following description: $h(0,j) = 0$, whereas $h(1, j) = j$, for any $0 \leq j \leq n$.
\end{proof}

\section{The homotopy colimit functor}

In this section we will consider the adjoint pair
\[
\xymatrix{
h_!: \sSets^\A \ar@<.5ex>[r] & \sSets/\NA : h^* \ar@<.5ex>[l]
}
\]
in more detail and prove Proposition \ref{prop:A} from the introduction. Recall that for a diagram $F$, the simplicial set $h_!(F)$ is defined as follows. Let $\int_{\A}F$ be the degreewise category of elements of $F$. This is a category object in $\sSets$ with an evident projection to $\A$. Then $h_!(F)$ is the diagonal of the nerve $N(\int_{\A}F)$, which is a bisimplicial set. Thus, an $n$-simplex of $h_!(F)$ is a pair $(A,x)$, where $A = (a_0 \rightarrow \cdots \rightarrow a_n)$ is an $n$-simplex of $\NA$ and $x \in F(a_0)_n$. Note that $h_!$ is compatible with the simplicial structures, in the sense that there is a natural isomorphism $h_!(M \otimes F) \simeq M \otimes h_!(F)$ for a simplicial set $M$ and a simplicial diagram $F$. \par 
For an object $X \rightarrow \NA$ of $\sSets/\NA$, the value of the right adjoint $h^*(X)$ may be described by
\begin{equation*}
h^*(X)(b) = \mathrm{Map}_{\NA}(N(b/\A), X)
\end{equation*}
where $b$ is any object of $\A$ and $\mathrm{Map}_{\NA}$ refers to the simplicial structure specified in Section \ref{sec:modelstructures}.

\begin{lemma}
\label{lem:h!cof}
The functor $h_!: \sSets^\A \rightarrow \sSets/\NA$ preserves monomorphisms.
\end{lemma}
\begin{proof}
This is clear from the explicit description of $h_!$.
\end{proof}
 
\begin{remark}
In fact, the functor $h_!$ has a left adjoint $h^+$, defined on representables $A: \Delta^n \rightarrow \NA$, $A = (a_0 \rightarrow \cdots \rightarrow a_n)$ by
\begin{equation*}
h^+(A) = \Delta^n \times \A(a_0, - ).
\end{equation*}
We will not use this left adjoint, since it is in general not left Quillen.
\end{remark}

\begin{lemma}
\label{lem:h*fib1}
The functor $h_!$ sends projective trivial cofibrations to covariant trivial cofibrations.
\end{lemma}
\begin{proof}
Observe that $h_!$ sends the generating trivial cofibration
\begin{equation*}
\Lambda_k^n \times \A(b,-) \longrightarrow \Delta^n \times \A(b,-)
\end{equation*}
to the map
\begin{equation*}
\Lambda_k^n \times N(b/\A)  \rightarrow \Delta^n \times N(b/\A),
\end{equation*}
which is a covariant trivial cofibration, since the covariant model structure is simplicial (see Section \ref{sec:modelstructures}).
\end{proof}

\begin{lemma}
\label{lem:h*fib2}
If $\A$ is a (generalized) Reedy category, the functor $h_!$ sends trivial Reedy cofibrations to covariant trivial cofibrations.
\end{lemma}
\begin{proof}
Let $\A^-(b,-)$ be as defined in Section \ref{subsec:reedy} and let $N^-(b/\A) = h_!\A^-(b,-)$. Then we have to verify that
\begin{equation*}
\Lambda_k^n \times N(b/\A) \cup \Delta^n \times N^-(b/\A) \longrightarrow \Delta^n \times N(b/\A)
\end{equation*}
is a covariant trivial cofibration. The map $N^-(b/\A) \rightarrow N(b/\A)$ is a monomorphism and so the pushout-product above is indeed a covariant trivial cofibration, again since the covariant model structure is simplicial.
\end{proof}

\begin{proof}[Proof of Proposition \ref{prop:A}]
The proposition follows directly from Lemmas \ref{lem:h!cof}, \ref{lem:h*fib1} and \ref{lem:h*fib2}.
\end{proof}

\begin{remark}
The Quillen pair of Proposition \ref{prop:A} is \emph{simplicial}, in the sense that there is a natural isomorphism
\begin{equation*}
\mathrm{Map}_{\NA}(h_!F, X) \simeq \mathrm{Map}(F, h^*X)
\end{equation*}
for any simplicial diagram $F$ and any simplicial set $X$ over $\NA$. Indeed, this is clear from the fact that $h_!$ strictly commutes with the simplicial structure, as remarked earlier.
\end{remark}

\section{The rectification functor}
\label{sec:rectification}

In this section we will prove Proposition \ref{prop:B} from the introduction, concerning the adjoint pair
\[
\xymatrix{
r_!: \sSets/\NA \ar@<.5ex>[r] & \sSets^\A : r^*. \ar@<.5ex>[l]
}
\]
For an $n$-simplex $\alpha$ in $\NA$ of the form
\[
\xymatrix{
a_0 \ar[r]^{f_1} & \cdots \ar[r]^{f_n} & a_n
}
\]
and an object $b$ in $\A$ we have
\begin{equation*}
r_!(\alpha)(b) = N(\alpha/b).
\end{equation*}
Here the category $\alpha/b$ has as objects pairs $(i, f)$, where $0 \leq i \leq n$ and $f: a_i \rightarrow b$ is a morphism in $\A$. Morphisms in $\alpha/b$ are commutative triangles
\[
\xymatrix{
a_i \ar[rr]^{f_{ij}}
\ar[dr]_f && a_j \ar[dl]^g \\
& b &
}
\] 
where $f_{ij} = f_j \circ \cdots \circ f_{i+1}$. The category $\alpha/b$ is in fact a poset and one easily verifies that $N(\alpha/b)$ is weakly equivalent to the (discrete) set $\A(a_0,b)$.  \par 
This description of $r_!$ also yields an explicit description of $r^*$ for any simplicial diagram $F$ on $\A$. Indeed, the $n$-simplices $x \in r^*(F)_\alpha$ over an $n$-simplex $\alpha$ in $\NA$ as above are families
\begin{equation*}
x = (x_u)_u, \quad u: \Delta^k \rightarrow \Delta^n,
\end{equation*}
where $x_u \in F(a_{u(k)})$ is a $k$-simplex, and these are compatible in the sense that for each commuting triangle on the left, the square on the right also commutes:
\[
\xymatrix{
\Delta^k\ar[r]^u\ar[d]_v & \Delta^n && \Delta^k \ar[d]_v\ar[r]^-{x_u} & F(a_{u(k)}) \ar[d]^{F(f_v)} \\
\Delta^{k'} \ar[ur]_{u'} &          && \Delta^{k'} \ar[r]^-{x_{u'}} & F(a_{u'(k')}).
}
\]
Here $f_v = f_{u(k)u'(k')}: a_{u(k)} \rightarrow a_{u'(k')}$ is the appropriate composition of $f_i$'s. So, said more informally, an $n$-simplex $x$ of $r^*(F)$ over a given $n$-simplex $\alpha$ of $\NA$ is an $n$-simplex $x_{\mathrm{id}} \in F(a_n)_n$ together with a compatible family of liftings of the faces of $x_{\mathrm{id}}$ to the appropriate $F(a_i)$, where $i$ is the final vertex of the face under consideration. In fact, a slightly smaller amount of data suffices to describe an $n$-simplex $x \in r^*(F)_\alpha$. Indeed, such an $x$ is completely determined by a sequence of simplices $x_i \in F(a_i)_i$, for $0 \leq i \leq n$, such that $f_i(x_{i-1}) = d_i x_i$.

We will prove Proposition \ref{prop:B} by the following two lemmas.

\begin{lemma}
\label{lem:r*fib1}
The functor $r^*: \sSets^\A \rightarrow \sSets/\NA$ preserves trivial fibrations.
\end{lemma}
\begin{proof}
Let $G \rightarrow F$ be a map in $\sSets^\A$. Given $\alpha = (a_0 \rightarrow \cdots \rightarrow a_n)$, an $n$-simplex in $\NA$ as above, one uses the explicit description of $r^*$ just given to check that a diagonal lifting in a square as on the left corresponds to a diagonal lifting on the right:
\[
\xymatrix{
\partial\Delta^n \ar[r]\ar[d] & r^*(G) \ar[d] && \partial\Delta^n \ar[d]\ar[r] & G(a_n)\ar[d] \\
\Delta^n\ar[r]\ar[dr]_\alpha & r^*(F) \ar[d] && \Delta^n \ar[r] & F(a_n) \\
& \NA && &
}
\]
Indeed, a priori a diagonal lifting on the left would also involve liftings of the various faces of $\Delta^n$ into the appropriate $G(a_i)$'s, but since all these faces are contained in $\partial\Delta^n$, such lifts are already provided by the top horizontal map. From this the lemma is clear.
\end{proof}

\begin{lemma}
\label{lem:r*fib2}
The functor $r^*: \sSets^\A \rightarrow \sSets/\NA$ sends projective fibrations to left fibrations.
\end{lemma}
\begin{proof}
The proof is identical to that of Lemma \ref{lem:r*fib1}, now using a correspondence between lifts in the same diagrams, but with $\partial\Delta^n$ replaced by $\Lambda_k^n$ for $0 \leq k < n$. Here we use that for $k < n$, all the faces that make up $\Lambda_k^n$ contain the final vertex $n$.
\end{proof}

\begin{remark}
Notice that the preceding proof does not work for the horns $\Lambda_n^n$. Indeed, already for $n=1$ and $\alpha: a_0 \rightarrow a_1$, a commutative square on the left corresponds to a vertex $y_1 \in G(a_1)$, together with a 1-simplex $x \in F(a_1)$ and a lift of $x_0 = d_1x$ to a vertex $\tilde x_0$ in $F(a_0)$, with $y_1$ mapped to $x_1 = d_0x$ by $G(a_1) \rightarrow F(a_1)$. A lifting $\Delta^1 \rightarrow r^*(G)$ would consist of a lift $y \in G(a_1)_1$ of $x$ (which can be found if $G(a_1) \rightarrow F(a_1)$ is a Kan fibration), together with a vertex $\tilde y_0$ in $G(a_0)$ over $y_0 \in G(a_1)_0$ lifting $\tilde x_0 \in F(a_0)_0$. In general, there is no way to obtain such a vertex $\tilde y_0$ from the data of a lift in the square on the right.
\end{remark}

\begin{proof}[Proof of Proposition \ref{prop:B}]
The proof follows from Lemmas \ref{lem:r*fib1} and \ref{lem:r*fib2}. Indeed, the first lemma ensures that $r_!$ preserves cofibrations. Using that the fibrant objects in the covariant model structure are the left fibrations over $\NA$ and the fibrations between such fibrant objects are also precisely maps that are left fibrations, the second lemma shows that $r^*$ preserves fibrant objects and fibrations between such. Since in any model category a cofibration is a weak equivalence if and only if it has the left lifting property with respect to fibrations between fibrant objects, this ensures that $r_!$ preserves trivial cofibrations.
\end{proof}

\begin{remark}
The pair $(r_!,r^*)$ is also simplicial. Indeed, it is immediate from our definition of $r_!$ that that for a simplicial set $X$ over $\NA$ and any $n \geq 0$, there is a natural isomorphism
\begin{equation*}
r_!(\Delta^n \otimes X) \simeq \Delta^n \otimes r_!(X)
\end{equation*}
and hence by adjunction for any simplicial diagram $F$ on $\A$ a natural isomorphism
\begin{equation*}
\mathrm{Map}_{\NA}(X, r^*F) \simeq \mathrm{Map}(r_!X, F).
\end{equation*}
\end{remark}

\section{Proof of the equivalence}
\label{sec:proof}

In this section we prove Theorem \ref{thm:c} stated in  the introduction. The proof will be split into two propositions.

\begin{proposition}
\label{prop:proof1}
For any simplicial diagram $F$, there exists a natural map $\tau: r_!h_!(F) \rightarrow F$, which is a weak equivalence whenever $F$ is projectively cofibrant.
\end{proposition}

\begin{proposition}
\label{prop:proof2}
For any simplicial set $X \rightarrow \NA$ over the nerve of $\A$, there exists a natural zigzag
\begin{equation*}
X \longrightarrow L(X) \longleftarrow h_!r_!(X)
\end{equation*}
of covariant weak equivalences over $\NA$.
\end{proposition}

\begin{proof}[Proof of Proposition \ref{prop:proof1}]
Recall that for an object $\pi: X \rightarrow \NA$ of $\sSets / \NA$, the value of the diagram $r_!(X)$ at an object $b$ of $\A$ is the simplicial set described by
\begin{equation*}
r_!(X)(b)_n = \{\, (x, \beta) \quad | \quad x \in X_n, \quad \beta: \pi(x_n) \rightarrow b \, \}, 
\end{equation*}
where $x_n$ denotes the final vertex of the $n$-simplex $x$. Thus, for a simplicial diagram $F$ on $\A$, the $n$-simplices of $r_!h_!(F)(b)$ are of the form
\[
\xymatrix@C=10pt{
(x \in F(a_0)_n,\, a_0 \ar[r]^-{\alpha_1} & \cdots \ar[r]^{\alpha_n} & a_n \ar[r]^{\beta} & b).
}
\]
The natural map $\tau: r_!h_!(F) \rightarrow F$ is defined by sending such a simplex to $F(\beta \alpha_n \cdots \alpha_1)(x)$, an $n$-simplex of $F(b)$. \par 
In order to prove that this map is a weak equivalence for cofibrant objects $F$, one can argue by the usual skeletal induction on $F$ and conclude that it suffices to prove that for each representable object $\Delta^n \times \A(a,-)$, the map
\begin{equation*}
\tau = \tau_{n,a}: r_!h_!(\Delta^n \times \A(a,-)) \longrightarrow \Delta^n \times \A(a,-)
\end{equation*}
is a projective weak equivalence. Now consider the square
\[
\xymatrix{
r_!h_!(\Delta^n \times \A(a,-)) \ar[r]^-{\tau_{n,a}} & \Delta^n \times \A(a,-) \\
r_!h_!(\Delta^0 \times \A(a,-)) \ar[u]\ar[r]_-{\tau_{0,a}} & \Delta^0 \times \A(a,-) \ar[u]
}
\]
where the vertical maps are given by the inclusion of the zero'th vertex and its image under $r_!h_!$. The right vertical map is a projective trivial cofibration and since $r_!$ and $h_!$ are both left Quillen, the left vertical map is so as well. So to prove that $\tau_{n,a}$ is a weak equivalence, it suffices to prove that $\tau_{0,a}$ is. \par 
But $r_!h_!(\Delta^0 \times \A(a,-))(b)$ is the nerve of the category $\A '(a,b)$ whose objects are factorizations $a \rightarrow c \rightarrow b$ of arrows in $\A(a,b)$, while its arrows are commutative diagrams
\[
\xymatrix@R=10pt@C=20pt{
& c \ar[dr]\ar[dd] & \\
a \ar[ur]\ar[dr] & & b. \\
& c' \ar[ur] &
}
\] 
The map $\tau$ is then the nerve of the composition functor $\A '(a,b) \rightarrow \A(a,b)$, where the latter set is interpreted as a discrete category. This map is a weak equivalence, because each connected component of $\A '(a,b)$ has an initial as well as a terminal object ($a = a \rightarrow b$ and $a \rightarrow b = b$ respectively). This proves the proposition.
\end{proof}

Before proving Proposition \ref{prop:proof2}, let us discuss some of the relevant functors. Consider a map of simplicial sets $\pi: X \rightarrow \NA$. Recall that $h_!r_!(X)$ is the simplicial set over $\NA$ having as $n$-simplices over $a_0 \rightarrow \cdots \rightarrow a_n$ the pairs
\begin{equation*}
(x, \, \pi(x_n) \rightarrow a_0 \rightarrow \cdots \rightarrow a_n)
\end{equation*}
where $x \in X_n$ and $x_n$ is the final vertex of $x$ (as before). The projection $h_!r_!(X) \rightarrow \NA$ maps such a pair to $a_0 \rightarrow \cdots \rightarrow a_n$. This functor can be defined over any simplicial set $B$, not necessarily of the form $\NA$. Indeed, for a map of simplicial sets $\pi: X \rightarrow B$ and an $n$-simplex $\beta$ of $B$, consider the simplicial set $X/\beta$: its $m$-simplices are pairs $(x, \xi)$, where $x$ is an $m$-simplex of $X$ and $\xi$ is an $(m+n+1)$-simplex of $B$, such that
\begin{equation*}
\pi(x) = \xi|_{\{0, \ldots, m\}} \quad \text{and} \quad \beta = \xi|_{\{m+1, \ldots, m+n+1\}}.
\end{equation*} 
Heuristically, we might denote such a simplex by 
\begin{equation*}
\pi(x_0) \rightarrow \cdots \rightarrow \pi(x_m) \rightarrow b_0 \rightarrow \cdots \rightarrow b_n.
\end{equation*}
As we let the simplex $\beta$ vary we obtain a bisimplicial set, which we denote $X//B$. Projecting to the simplex $\beta$ gives a map
\begin{equation*}
\mathrm{diag}(X//B) \longrightarrow B.
\end{equation*}
Observe that if $B$ is $\NA$, then there is a natural isomorphism
\begin{equation*}
h_!r_!(X) \simeq \mathrm{diag}(X // \NA).
\end{equation*}
This more general description of $h_!r_!$ will be of use to us in Section \ref{sec:QuillenA}. Observe that the functor $\mathrm{diag}(- //B)$ preserves colimits and monomorphisms. In fact it is left Quillen, as will follow from our proofs.

We now define the zigzag of the proposition. For a simplicial set $B$, consider the map $B^{\Delta^1} \rightarrow B$ given by evaluating at 0 and define a functor $L: \mathbf{sSets}/B \rightarrow \mathbf{sSets}/B$ by
\begin{equation*}
L(X) = X \times_{B} B^{\Delta^1}.
\end{equation*}
The map $L(X) \rightarrow B$ is determined by the evaluation $B^{\Delta^1} \rightarrow B$ at 1. In the special case where $B = \NA$, the $n$-simplices of $L(X)$ can be described as pairs $(x, \lambda)$, where $x$ is an $n$-simplex of $X$ and $\lambda$ is a commutative diagram in $\A$ of the form
\[
\xymatrix{
\pi(x_0) \ar[r]\ar[d]_{\lambda_0} & \pi(x_1) \ar[r]\ar[d]_{\lambda_1} & \cdots  \ar[r] & \pi(x_n) \ar[d]_{\lambda_n} \\
a_0 \ar[r] & a_1 \ar[r] & \cdots \ar[r] & a_n
}
\]
and the map to $\NA$ is given by taking the bottom row. We will refer to such diagrams as \emph{ladders}. Clearly, the functor $L$ preserves colimits and monomorphisms. We will see later that it is in fact left Quillen.

For a map $\pi: X \rightarrow B$, let us define natural maps
\[
\xymatrix{
X \ar[r]^-{\iota} & L(X) & \mathrm{diag}(X//B) \ar[l]_-{\gamma}
}
\]
over $B$. The map $\iota$ is the `constant path' map, which sends an $n$-simplex $x$ to the pair $(x, s_0\pi(x))$, with $s_0$ denoting the pullback along the degeneracy map $\Delta^1 \rightarrow \Delta^0$. Now consider the map $g: \Delta^1 \times \Delta^n \longrightarrow \Delta^{2n+1}$ that is determined by the requirement that $g(a,b) = (n+1)a + b$, for $a = 0,1$ and $0 \leq b \leq n$. Then pullback along $g$ defines a map $\mathrm{diag}(B//B) \rightarrow B^{\Delta^1}$, which in turn induces the map $\gamma$. In particular, if $B = \NA$, the map $\gamma$ sends a simplex
\begin{equation*}ß
(x, \, \pi(x_n) \rightarrow a_0 \rightarrow \cdots \rightarrow a_n)
\end{equation*}
to the pair given by $x$ and the ladder
\[
\xymatrix{
\pi(x_0) \ar[r]\ar[d] & \pi(x_1) \ar[r]\ar[d] & \cdots  \ar[r] & \pi(x_n) \ar[d] \\
a_0 \ar[r] & a_1 \ar[r] & \cdots \ar[r] & a_n
}
\]
in which the vertical maps  $\pi(x_i) \rightarrow a_i$ are the appropriate compositions of arrows formed from the sequence
\begin{equation*}
\pi(x_0) \rightarrow \cdots \rightarrow \pi(x_n) \rightarrow a_0 \rightarrow \cdots \rightarrow a_n.
\end{equation*}

Proposition \ref{prop:proof2} is now a consequence of the following:

\begin{proposition}
\label{prop:zigzag}
Let $B = \NA$. For any simplicial set $X$ over $B$, the maps $\iota$ and $\gamma$ in the zigzag
\begin{equation*}
X \xrightarrow{\iota} L(X) \xleftarrow{\gamma} \mathrm{diag}(X//B)
\end{equation*} 
are covariant weak equivalences over $B$.
\end{proposition}

\begin{proof}
It is easy to see that the map $\iota$ is part of a covariant deformation retract. The retraction $L(X) \rightarrow X$ is given by evaluating at 0. Indeed, consider the constant path map $\mathrm{con}: B \rightarrow B^{\Delta^1}$ and the evaluation map $\mathrm{ev}_0: B^{\Delta^1} \rightarrow B$. There is an evident homotopy from $\mathrm{con} \circ \mathrm{ev}_0$ to the identity on $B^{\Delta^1}$, which also yields the necessary homotopy from the composite $L(X) \rightarrow X \rightarrow L(X)$ to the identity on $L(X)$.

We now wish to show that $\gamma$ is a covariant weak equivalence. First of all, notice that by the usual skeletal induction on $X$, it suffices to prove this for a representable of the form $\Delta^n \rightarrow B$. Indeed, if $X'$ is obtained from $X$ by attaching an $n$-simplex $\Delta^n \rightarrow B$, consider the cube over $B$
\[
\xymatrix{
\partial \Delta^n \ar[dd]\ar[rr]\ar[dr] & & X\ar'[d][dd]\ar[dr] & \\
& \mathrm{diag}(\partial \Delta^n//B) \ar[rr]\ar[dd] & & \mathrm{diag}(X//B) \ar[dd] \\
\Delta^n \ar'[r][rr]\ar[dr] & & X' \ar[dr] & \\
& \mathrm{diag}(\Delta^n//B) \ar[rr] & & \mathrm{diag}(X'//B)
}
\]
in which the back and front faces are pushouts. Since the vertical maps are cofibrations ($\mathrm{diag}(-//B)$ preserves monos) and all objects are cofibrant, these pushouts are also homotopy pushouts. It follows that the component $X' \rightarrow \mathrm{diag}(X'//B)$ of $\gamma$ from back to front is a weak equivalence whenever the other three components from back to front are weak equivalences, which establishes the inductive step. 

So now consider a simplex $\beta: \Delta^n \rightarrow B$ and the diagram over $B$
\[
\xymatrix{
b_0 \ar[r]^-\iota\ar[d] & L(b_0) \ar[d] & \mathrm{diag}(b_0//B) \ar[d]\ar[l]_-\gamma \\
\beta \ar[r]^-\iota & L(\beta) & \mathrm{diag}(\beta//B)\ar[l]_-\gamma,
}
\]
where $b_0$ is shorthand for the map $\Delta^0 \rightarrow B$ given by the initial vertex of $\beta$. The map $b_0 \rightarrow \beta$ is part of a covariant deformation retract by Lemma \ref{lem:leftanod2}, so that all the vertical maps are in fact weak equivalences; indeed, the middle one is a weak equivalence since we have proved $L$ is weakly equivalent to the identity functor, and the right one is an equivalence since in the specific case $B = \NA$ we have $\mathrm{diag}(-//B) = h_!r_!$, which is a left Quillen functor. To prove that $\mathrm{diag}(\beta//B) \rightarrow L(\beta)$ is a weak equivalence, it thus suffices to prove that the map $\mathrm{diag}(b_0//B) \rightarrow L(b_0)$ is a weak equivalence. Note that the top row of the previous diagram may be completed to a commutative triangle
\[
\xymatrix{
b_0 \ar[d]_i\ar[dr] & \\
\mathrm{diag}(b_0//B) \ar[r] & L(b_0).
}
\]
We already know that the diagonal map is a weak equivalence. It thus suffices to show that the vertical map $i$ is a weak equivalence. We will show that it is part of a covariant deformation retract, very much analogous to the proof of Lemma \ref{lem:leftanod1}. Indeed, let $r: \mathrm{diag}(b_0//B) \rightarrow b_0$ be the unique such map and define $h: \Delta^1 \times \mathrm{diag}(b_0//B) \rightarrow \mathrm{diag}(b_0//B)$ as follows. Consider an $n$-simplex $(\alpha, \xi)$ of $\Delta^1 \times \mathrm{diag}(b_0//B)$. Let $k$ be the minimal integer such that $\alpha(k) = 1$, if it exists. We may heuristically draw such a simplex as 
\[
\xymatrix@R=5pt{
b_0 \ar@{=}[r] & \cdots \ar@{=}[r] & b_0 \ar[r] & x_0 \ar[r] & \cdots \ar[r] & x_{k-1} \ar[r] & x_k \ar[r] & \cdots \ar[r] & x_n, \\
&&& 0 & \cdots & 0 & 1 & \cdots & 1,
}
\] 
with $n$ copies of $b_0$ occurring at the start. Then $h(\alpha, \xi)$ is the $n$-simplex represented by the picture
\[
\xymatrix@R=5pt{
b_0 \ar@{=}[r] & \cdots \ar@{=}[r] & b_0 \ar@{=}[r] & b_0 \ar@{=}[r] & \cdots \ar@{=}[r] & b_0 \ar[r] & x_k \ar[r] & \cdots \ar[r] & x_n,
}
\] 
more formally defined as a $k$-fold degeneracy of a $k$-fold face of the $(2n+1)$-simplex drawn above. It is immediately verified that this $h$ satisfies the necessary properties, completing the proof.
\end{proof}

We end this section by placing Proposition \ref{prop:zigzag} in a more general context. Indeed, the reader should note that our proof hardly uses the fact that $B$ is the nerve of a category at all. To be specific, we also have the following more general result:

\begin{proposition}
\label{prop:zigzag2}
Let $B$ be an $\infty$-category (i.e. a fibrant object in the Joyal model structure). Then for any simplicial set $X$ over $B$, the maps $\iota$ and $\gamma$ in the zigzag
\begin{equation*}
X \xrightarrow{\iota} L(X) \xleftarrow{\gamma} \mathrm{diag}(X//B)
\end{equation*} 
are covariant weak equivalences over $B$.
\end{proposition}
\begin{proof}
The proof of Proposition \ref{prop:zigzag} carries over verbatim, except for the claim that $\mathrm{diag}(b_0//B) \rightarrow \mathrm{diag}(\beta//B)$ is a covariant weak equivalence over $B$, for $\beta$ a simplex of $B$ with initial vertex $b_0$. We will prove this claim as Lemma \ref{lem://leftQ}.
\end{proof}

\section{A localization of the projective model structure}

In this section we prove Corollary \ref{prop:localized} from the introduction. Denote by $S$ the set of natural transformations $\A(b,-) \rightarrow \A(a,-)$ induced by morphisms $a \rightarrow b$ in $\A$. It follows formally from Theorem \ref{thm:c} that $h_!$ induces a Quillen equivalence between the localization of the projective model structure with respect to $S$ and the localization of the covariant model structure with respect to $h_!S$. Denote by $T$ the class of weak equivalences in the latter (localized) model structure.

\begin{lemma}
\label{lem:h!localized}
The left Bousfield localization of the covariant model structure on $\sSets/\NA$ with respect to $T$ is the same as the localization of the covariant model structure with respect to all the maps of the form
\[
\xymatrix{
\Delta^0 \ar[rr]^1 \ar[dr] & & \Delta^1 \ar[dl] \\
& \NA, &
}
\]
where $1$ indicates the inclusion of the final vertex of $\Delta^1$.
\end{lemma}
\begin{proof}
For a morphism $f: a \rightarrow b$ in $\A$, consider the diagram
\[
\xymatrix{
\Delta^0 \ar[d]_{\mathrm{id}_b}\ar[r]^{1} & \Delta^1 \ar[dr]_{a/f} & \Delta^0 \ar[l]_{0}\ar[d]^{\mathrm{id}_a} \\
N(b/\A) \ar[rr]_{f^*} \ar[dr] & & N(a/\A) \ar[dl] \\ 
& \NA &
}
\]
where $a/f$ denotes the 1-simplex of $N(a/\A)$ given by
\[
\xymatrix{
& a \ar@{=}[dl] \ar[dr]^f & \\
a \ar[rr]_f & & b.
}
\]
The map $0$ is a weak equivalence in the covariant model structure, and so are the two vertical maps, by Lemma \ref{lem:leftanod1}. Therefore $a/f$ is a weak equivalence in the covariant model structure as well. By two-out-of-three, it then follows that localizing with respect to $f^*$ is the same as localizing with respect to the map $1$, and the lemma follows.
\end{proof}

\begin{proof}[Proof of Corollary \ref{prop:localized}]
It remains to show that any weak equivalence of simplicial sets over $\NA$ (in the Kan-Quillen model structure) is contained in $T$. Given the definition of the covariant model structure, it will suffice to show that any map in $\sSets/\NA$ of the form
\[
\xymatrix{
\Lambda_n^n \ar[rr]\ar[dr] & & \Delta^n \ar[dl] \\
& \NA & 
}
\]
is contained in $T$. In fact, we will show the following: for any $V \subseteq \Lambda_n^n$ a union of faces, all containing the final vertex $\{n\}$, the two horizontal maps
\[
\xymatrix{
\Delta^0 \ar[dr]\ar[r]^n & V \ar[d]\ar[r] & \Delta^n \ar[dl]^\alpha\\
& \NA & 
}
\]
are contained in $T$ (where the first horizontal map is the inclusion of the final vertex $\{n\}$). We will work by induction on $n$. In case $n=1$, the 1-simplex $\alpha$ corresponds to a morphism $f: a \rightarrow b$ in $\A$ and we have to show that the inclusion
\[
\xymatrix{
\Delta^0 \ar[rr]^1\ar[dr] & & \Delta^1 \ar[dl] \\
& \NA & 
}
\]
is in $T$. This follows from Lemma \ref{lem:h!localized} above. If $n > 1$, we use a further induction on the number of faces in $V$. If $V$ consists of one face $d_i\Delta^n$, for $i \neq n$, then the inclusion $\{n\} \rightarrow V$ is in $T$ by induction. Over the simplex $\alpha: \Delta^n \rightarrow \NA$ we have a diagram
\[
\xymatrix{
\Delta^0 \ar[d]_{n} \ar[r]^1 & \Delta^1 \ar[d]_{(0n)} & \Delta^0 \ar[l]_0 \ar[dl]^0 \\
V \ar[r] & \Delta^n. & 
}
\]
The maps $1$ and $n$ are in $T$ by induction and both the maps labelled $0$ are in $T$ by Lemma \ref{lem:leftanod2}. By two-out-of-three, the maps $(0n)$ and $V \rightarrow \Delta^n$ are then also in $T$. Now suppose $V$ consists of more than one face and write $V = W \cup d_i\Delta^n$, where $W$ consists of one fewer face and $i \neq n$. Over $\alpha: \Delta^n \rightarrow \NA$ we can now form the following diagram:
\[
\xymatrix{
\Delta^0 \ar[r]^-n\ar[dr]_n & W \cap d_i\Delta^n \ar[d]^\gamma\ar[r]^-\beta & W \ar[d]^\delta \ar[dr]^\zeta & \\
& d_i\Delta^n \ar[r]_\varepsilon & V \ar[r]_\eta & \Delta^n.
}
\]
The simplicial set $W \cap d_i\Delta^n$ is a face of $d_i\Delta^n$, so that both maps labelled $n$ are in $T$ by the inductive hypothesis (and therefore also $\gamma$ is in $T$). The composite $\beta \circ n$ of the top two horizontal maps is in $T$ by the inductive hypothesis on $W$, so that $\beta$ too is in $T$. Since the square is a pushout, it follows that $\delta$ and $\varepsilon$ are in $T$. Again $\zeta$ is in $T$ by induction, so that finally $\eta$ must be in $T$ as well.
\end{proof}

\section{Homotopy invariance of the covariant model structure}

The aim of this section is to prove Proposition \ref{prop:invariance}. Let $f: X \longrightarrow Y$ be a map of simplicial sets, which is a categorical equivalence (i.e. a weak equivalence in the Joyal model structure). We wish to show that the adjoint pair 
\[
\xymatrix{
\mathbf{sSets}/X \ar@<.5ex>[r]^{f_!} & \mathbf{sSets}/Y \ar@<.5ex>[l]^{f^*},
}
\]
which is easily seen to be a Quillen adjunction, is in fact a Quillen equivalence. First, form a commutative diagram
\[
\xymatrix{
X \ar[d]\ar[r]^f & Y \ar[d] \\
X' \ar[r] & Y',
}
\]
in which the vertical maps are inner anodyne (i.e. are in the saturated class generated by the inner horn inclusions) and $X'$ and $Y'$ are quasicategories. Such a diagram can be obtained by choosing an inner anodyne $X \rightarrow X'$, with $X'$ fibrant, and then choosing an inner anodyne $Y \cup_X X' \rightarrow Y'$ with $Y'$ fibrant. Note that all the maps in this diagram are categorical equivalences. To prove Proposition \ref{prop:invariance}, it will thus suffice to prove the following two lemmas:

\begin{lemma}
\label{lem:invquasicats}
Proposition \ref{prop:invariance} holds when $f$ is a categorical equivalence between quasicategories.
\end{lemma}

\begin{lemma}
\label{lem:invinneranod}
Proposition \ref{prop:invariance} holds when $f$ is inner anodyne.
\end{lemma}

\begin{proof}[Proof of Lemma \ref{lem:invquasicats}]
It is easy to see that Proposition \ref{prop:invariance} holds when $f$ is a trivial fibration, by checking that the derived unit and counit of the adjoint pair $(f_!, f^*)$ are weak equivalences. Applying Brown's lemma \cite{hirschhorn}, we deduce that Proposition \ref{prop:invariance} holds when $f$ is a weak equivalence between fibrant objects. Indeed, for a weak equivalence $f: X \rightarrow Y$ between fibrants, this lemma states that there exists a diagram
\[
\xymatrix{
X \ar[rr]^f\ar@<.5ex>[dr]^i & & Y \\
& Z \ar[ur]_h \ar@<.5ex>[ul]^g& 
}
\]
in which $hi = f$, $gi = \mathrm{id}_X$ and both the maps $g$ and $h$ are trivial fibrations. Since the adjunctions associated to $g$ and $h$ are Quillen equivalences, we conclude first that the one associated to $i$ is and then that the one for $f = hi$ is.
\end{proof}

To prove Lemma \ref{lem:invinneranod} we will use the following two lemmas, which we prove later in this section:

\begin{lemma}
\label{lem:pullbackinner}
Consider a pullback square of simplicial sets
\[
\xymatrix{
X \times_Y Z \ar[r]^-g\ar[d] & Z \ar[d]^p \\
X \ar[r]_f & Y
}
\]
in which $f$ is inner anodyne and $p$ is a left fibration. Then $g$ is a trivial cofibration in the Joyal model structure.
\end{lemma}

\begin{lemma}
\label{lem:extendleftfib}
Let $0 < k < n$ and let $p: A \longrightarrow \Lambda_k^n$ be a left fibration. Then there exists a left fibration $q: B \longrightarrow \Delta^n$ and an equivalence
\[
\xymatrix{
A \ar[rr]^-g\ar[dr]_p && \Lambda_k^n \times_{\Delta^n} B \ar[dl] \\
& \Lambda_k^n
}
\]
in the covariant model structure over $\Lambda_k^n$.
\end{lemma}

\begin{remark}
\label{rmk:covariantwe}
In the previous lemma, we might as well have stated that $g$ is an equivalence in the Joyal model structure or that $g$ induces weak homotopy equivalences on the fibers over each vertex of $\Lambda_k^n$; indeed, these are all equivalent formulations of what it means for a map to be an equivalence between fibrant objects in the covariant model structure.
\end{remark}

In the proof of Lemma \ref{lem:invinneranod}, we will need the notion of a \emph{minimal} inner fibration \cite{joyal, htt}. Recall that an inner fibration $p: V \rightarrow X$ is minimal if $g = g'$ for every pair of maps $g, g': \Delta^n \rightarrow V$ which are $J$-homotopic relative to $\partial\Delta^n$ over $X$. For any inner fibration $q: W \rightarrow X$, there exists a commutative diagram
\[
\xymatrix{
V \ar[r]^i \ar[dr]_p & W \ar[r]^r \ar[d]_q & V \ar[dl]^p \\
& X &
}
\]
exhibiting $V$ as a retract of $W$ (so that $ri = \mathrm{id}_V$) in which $p: V \rightarrow X$ is a minimal inner fibration and $i$ is an equivalence in the Joyal model structure. Note that if $q$ happened to be a left fibration, then $p$, being a retract of $q$, is a left fibration as well. Furthermore, the pullback of any minimal fibration is minimal. Finally, if a map between minimal inner fibrations over $X$ is a $J$-homotopy equivalence (over $X$), then it is necessarily an isomorphism.

\begin{proof}[Proof of Lemma \ref{lem:invinneranod}]
First we show that the derived counit $\mathbf{L}f_!\mathbf{R}f^* \rightarrow \mathrm{id}$ is an equivalence (or, equivalently, that $\mathbf{R}f^*$ is a fully faithful functor). To this end, consider a left fibration $q: B \longrightarrow Y$ and form the following pullback square:
\[
\xymatrix{
X \times_Y B \ar[r]^-g \ar[d]_{f^*q} & B \ar[d]^q \\
X \ar[r]_f & Y. 
}
\]
We should check that $g$ is a covariant weak equivalence over $Y$, which is guaranteed by Lemma \ref{lem:pullbackinner} and the fact that the covariant model structure is a Bousfield localization of the Joyal model structure.

It remains to verify that $\mathbf{R}f^*$ is essentially surjective, i.e. that each left fibration over $X$ is (up to covariant equivalence) the pullback along $f$ of a left fibration over $Y$. Let $\mathcal{A}$ be the class of morphisms $f$ such that $\mathbf{R}f^*$ is essentially surjective. By Lemma \ref{lem:extendleftfib}, the class $\mathcal{A}$ contains all inner horn inclusions. If we show that $\mathcal{A}$ is closed under retracts, transfinite compositions and pushouts, it follows that it contains all inner anodyne maps. The first two are straightforward to deal with; let us focus on the case of pushouts. Consider a pushout square
\[
\xymatrix{
X \ar[r]^f \ar[d]_g & Y \ar[d] \\
X' \ar[r]_h & Y'
}
\]
where $f$ is an inner anodyne contained in $\mathcal{A}$. For any left fibration $p': A' \rightarrow X'$ we should construct a corresponding left fibration $q': B' \rightarrow Y'$ and a covariant equivalence between $p'$ and $h^*q'$. Without loss of generality we may assume $p'$ is a minimal fibration. Then the pullback $g^*p'$ is a minimal left fibration over $X$ and, by the assumption that $f$ is in $\mathcal{A}$, there exists a left fibration $q: B \rightarrow Y$ and a covariant weak equivalence $g^*p' \rightarrow f^*q$ over $X$. We may replace $q$ by a minimal left fibration $q_0: B_0 \rightarrow Y$; in particular, there is a map $q \rightarrow q_0$ which is a weak equivalence in the Joyal model structure over $Y$ and therefore a covariant weak equivalence over $Y$. Since $f^*$ is right Quillen, the map $f^*q \rightarrow f^*q_0$ is a covariant weak equivalence over $X$. The resulting map $g^*p' \rightarrow f^*q_0$ is a weak equivalence between left fibrations over $X$ and therefore a $J$-homotopy equivalence; since both are minimal fibrations, it is in fact an isomorphism. It follows that both squares in the diagram
\[
\xymatrix{
A' \ar[d]_{p'} & X \times_{X'} A \ar[d]_{g^*p'}\ar[l]\ar[r] & B_0 \ar[d]^{q_0} \\
X' & X \ar[l]^g \ar[r]_f & Y
}
\]
are pullbacks of simplicial sets. We may now form the pushouts of the bottom and top rows to obtain the desired map $q': B' \rightarrow Y'$. Indeed, to check that $q'$ is a left fibration, it suffices to verify the necessary lifting property after pulling back to $X'$ and $Y$. The pullback of $q'$ to $X'$ is (isomorphic to) $p'$, whereas its pullback to $Y$ is (isomorphic to) $q_0$, finishing the proof. 
\end{proof}

We conclude this section with the proofs of Lemmas \ref{lem:pullbackinner} and \ref{lem:extendleftfib}. First we need another auxiliary result. Consider the linear order $[n]$ with objects $\{0, \ldots, n\}$ and view it as a Reedy category where each nonidentity map decreases degree. Then a functor $F: [n] \rightarrow \mathbf{sSets}$ is Reedy fibrant if and only if $F(i)$ is a Kan complex for each $0 \leq i \leq n$ and $F(i) \rightarrow F(i+1)$ is a Kan fibration for each $0 \leq i \leq n-1$. Furthermore, any $F$ is Reedy cofibrant.

\begin{lemma}
\label{lem:h!fibrant}
If $F: [n] \rightarrow \mathbf{sSets}$ is Reedy fibrant, then $h_! F$ is a left fibration over $\Delta^n$.
\end{lemma}  
\begin{proof}
We first show that $h_!F$ is an inner fibration over $\Delta^n$. Consider a lifting problem
\[
\xymatrix{
\Lambda_k^m \ar[d]\ar[r]^f & h_!F \ar[d] \\
\Delta^m \ar[r]_{\varphi} \ar@{-->}[ur] & \Delta^n 
}
\]
for $0 < k < m$. Write $\Lambda_{0,k}^m$ for the union of the faces $d_i \Delta^m$ for $i \neq 0,k$. Then the map $f$ corresponds to a diagram
\[
\xymatrix{
\Lambda_{0,k}^m \ar[d]\ar[r] & F(\varphi(0)) \ar[d] \\
\Lambda_k^m \ar[r] & F(\varphi(1)).
}
\]
The right vertical map is a Kan fibration by assumption, whereas the left vertical map is a trivial cofibration in the Kan-Quillen model structure; indeed, both $\Lambda_{0,k}^m$ and $\Lambda_k^m$ are weakly contractible. Therefore, there exists a lifting $\Lambda_k^m \rightarrow F(\varphi(0))$. Since $F(\varphi(0))$ is a Kan complex, there exists a further extension of this map to a map $\Delta^m \rightarrow F(\varphi(0))$, which provides the necessary lift in our original diagram.

To see that $h_!F$ is a left fibration, consider a lifting problem
\[
\xymatrix{
\Lambda_0^m \ar[d]\ar[r]^f & h_!F \ar[d] \\
\Delta^m \ar[r]_{\varphi} \ar@{-->}[ur] & \Delta^n. 
}
\]
Now $f$ corresponds simply to a map $\Lambda_0^m \rightarrow F(\varphi(0))$, which admits an extension to a map $\Delta^m \rightarrow F(\varphi(0))$ by the assumption that $F(\varphi(0))$ is a Kan complex. Such an extension solves the lifting problem.
\end{proof}

\begin{proof}[Proof of Lemma \ref{lem:pullbackinner}]
Denote by $\mathcal{A}$ the class of monomorphisms $i: A \rightarrow B$ of simplicial sets having the property that for any left fibration $W \rightarrow B$, the pullback of $i$ along this fibration is a trivial cofibration in the Joyal model structure. Since the class of trivial cofibrations is saturated and pullbacks in simplicial sets distribute over colimits, the class $\mathcal{A}$ is saturated as well. Therefore it suffices to show that $\mathcal{A}$ contains all inner horn inclusions $i: \Lambda_k^n \rightarrow \Delta^n$, for $0 < k < n$. Note that we only need to prove that the pullback of such a map is a weak equivalence, since it is automatically a mono and hence a cofibration. 

Consider a left fibration $p: W \rightarrow \Delta^n$. By part (b) of Proposition \ref{prop:A}, the functor $h^*p$ is Reedy fibrant, so that $h_!h^*p$ is a left fibration over $\Delta^n$ by Lemma \ref{lem:h!fibrant}. Furthermore, Theorem \ref{thm:c} guarantees that the counit $h_!h^*p \rightarrow p$ is a covariant weak equivalence over $\Delta^n$ and, since both domain and codomain are fibrant, a Joyal weak equivalence. It follows that the map
\begin{equation*}
\Lambda_k^n \times_{\Delta^n} h_!h^*p \longrightarrow \Lambda_k^n \times_{\Delta^n} p
\end{equation*}
is a covariant weak equivalence between left fibrations over $\Lambda_k^n$ and hence a Joyal weak equivalence as well. Now form the following square:
\[
\xymatrix{
\Lambda_k^n \times_{\Delta^n} h_!h^*p \ar[r]\ar[d] & \Lambda_k^n \times_{\Delta^n} p \ar[d] \\
h_!h^*p \ar[r] & p.
}
\]
We wish to show that the right vertical map is a Joyal weak equivalence; by two-out-of-three, it suffices to show that the left vertical map is a Joyal weak equivalence. But for any functor $F: [n] \rightarrow \mathbf{sSets}$, the map $\Lambda_k^n \times_{\Delta^n} h_!F \rightarrow h_!F$ is a pushout of the map $\Lambda_k^n \times F(0) \rightarrow \Delta^n \times F(0)$ and hence inner anodyne.
\end{proof}

In the proof of Lemma \ref{lem:extendleftfib} we will need another pair of lemmas:

\begin{lemma}
\label{lem:finalfiber}
Let $p: B \rightarrow \Delta^n$ be a left fibration of simplicial sets and write $B_n$ for the fiber of $B$ over the vertex $n$. Then the inclusion $B_n \rightarrow B$ is a weak homotopy equivalence of simplicial sets.
\end{lemma}
\begin{proof}
Consider a fibrant replacement of the functor $r_!(p)$, i.e. a natural transformation $r_!(p) \rightarrow F$ which is a weak homotopy equivalence at every object and such that $F$ takes values in Kan complexes. Then, by Theorem \ref{thm:c}, the map $p \rightarrow r^* F$ is a covariant weak equivalence between left fibrations; in particular, the map of fibers $B_n \rightarrow (r^*F)_n = F(n)$ is a weak homotopy equivalence. This map factors as
\begin{equation*}
B_n \longrightarrow r_!(p)(n) \longrightarrow F(n)
\end{equation*}
where the second map is a weak homotopy equivalence by assumption. By definition, $r_!(p)(n)$ is precisely the simplicial set $B$ and the conclusion follows by two-out-of-three.
\end{proof}

\begin{lemma}
\label{lem:finalfiber2}
Let $V \subseteq \Delta^n$ be a simplicial subset of the $n$-simplex of one of the following two forms:
\begin{itemize}
\item[(a)] $V$ is a union of faces $d_i\Delta^n$, each of which contains the final vertex $n$ of $\Delta^n$. In other words, $V$ is a union of faces contained in $\Lambda_n^n$.
\item[(b)] $V$ is an inner horn $\Lambda_k^n$ for $0 < k < n$.
\end{itemize}
If $B \rightarrow V$ is a left fibration of simplicial sets, then the inclusion $B_n \rightarrow B$ is a weak homotopy equivalence.
\end{lemma}
\begin{proof}
We prove both statements by induction on $n$. When $n = 1$, the only case is where $V$ is the vertex $1$ in the simplex $\Delta^1$ and the statement is trivial. For higher $n$, first consider $V \subseteq \Delta^n$ of type (a). Write $V = F_0 \cup \cdots \cup F_m$, where the various $F_m$ are faces of $\Delta^n$ containing the vertex $n$. Write $B_{F_i}$ for the pullback of $B$ to a face $F_i$. We will show that each of the inclusions
\begin{equation*}
B_n \rightarrow B_{F_0} \longrightarrow B_{F_0} \cup B_{F_1} \longrightarrow \cdots \longrightarrow B_{F_0} \cup \cdots \cup B_{F_m} = B
\end{equation*}
is a weak homotopy equivalence. The first map $B_n \rightarrow B_{F_0}$ is a weak equivalence by an application of Lemma \ref{lem:finalfiber} to the left fibration $B_{F_0} \rightarrow F_0$. For any of the subsequent inclusions, there is a pushout square as follows:
\[
\xymatrix{
B_{F_0} \cup \cdots \cup B_{F_{i-1}} \ar[r] & B_{F_0} \cup \cdots \cup B_{F_i} \\
B_{F_0 \cap F_i} \cup \cdots \cup B_{F_{i-1} \cap F_i} \ar[r]^-{\ast}\ar[u] & B_{F_i}. \ar[u]
}
\]
We claim that the map marked by an asterisk is a trivial cofibration in the Kan-Quillen model structure, so that the top map is a trivial cofibration as well. Indeed, each of the intersections $F_j \cap F_i$ is a face of $F_i$ containing its final vertex $n$, so that by the inductive hypothesis we may apply case (a) of the lemma to conclude that the map
\begin{equation*}
B_n \rightarrow B_{F_0 \cap F_i} \cup \cdots \cup B_{F_{i-1} \cap F_i}
\end{equation*}
is a weak equivalence. Since $B_n \rightarrow B_{F_i}$ is a weak equivalence by Lemma \ref{lem:finalfiber}, it follows that $\ast$ is a weak equivalence by two-out-of-three. 

For case (b) of the lemma, we again consider all the faces making up $\Lambda_k^n$ one by one, giving a sequence of inclusions 
\begin{equation*}
B_n \rightarrow B_{d_0\Delta^n} \longrightarrow \cdots \longrightarrow \bigcup_{i \neq k} B_{d_i\Delta^n} = B.
\end{equation*}
We will prove that each of these maps is a weak homotopy equivalence. This is done in exactly the same way as for case (a), except for the final inclusion
\begin{equation*}
\bigcup_{i \neq k, n} B_{d_i\Delta^n} \longrightarrow \bigcup_{i \neq k} B_{d_i\Delta^n}.
\end{equation*}
This map is a pushout of the map
\begin{equation*}
\bigcup_{i \neq k, n} B_{d_id_n\Delta^n} \longrightarrow B_{d_n \Delta^n},
\end{equation*}
which we claim to be a weak homotopy equivalence. Indeed, the inclusion
\begin{equation*}
\bigcup_{i \neq k, n} d_i d_n\Delta^n \longrightarrow d_n\Delta^n
\end{equation*}
is isomorphic to the horn inclusion $\Lambda_k^{n-1} \subseteq \Delta^{n-1}$. Therefore 
\begin{equation*}
B_{n-1} \longrightarrow \bigcup_{i \neq k, n} B_{d_id_n\Delta^n}
\end{equation*}
is a weak homotopy equivalence by applying the inductive hypothesis to case (a) (for $k = n-1$) or case (b) (for $k < n-1$). Again, $B_{n-1} \rightarrow B_{d_n\Delta^n}$ is a weak equivalence by Lemma \ref{lem:finalfiber}, so that the conclusion follows by two-out-of-three.
\end{proof}

\begin{proof}[Proof of Lemma \ref{lem:extendleftfib}]
Let $p: A \rightarrow \Lambda_k^n$ be a left fibration and write $j: \Lambda_k^n \rightarrow \Delta^n$ for the inner horn inclusion. Consider the functor
\begin{equation*}
r_!(jp): [n] \longrightarrow \mathbf{sSets}
\end{equation*} 
and a fibrant replacement $r_!(jp) \rightarrow F$. Then write $q: B \rightarrow \Delta^n$ for the left fibration $r^*F$ and note that by construction there is a map $g: p \rightarrow j^*q$. To check that $g$ is a covariant weak equivalence, it suffices (since $p$ and $j^*q$ are both left fibrations) to check that it induces a weak homotopy equivalence on the fibers over vertices of $\Lambda_k^n$. For $0 \leq i \leq n$, the fiber of $j^*q$ over the vertex $i$ is precisely the simplicial set $F(i)$, and the map on fibers induced by $g$ factors as follows:
\begin{equation*}
A_i \longrightarrow r_!(jp)(i) \longrightarrow F(i).
\end{equation*}
The second map is a weak homotopy equivalence by assumption. The middle simplicial set is given by the formula
\begin{equation*}
r_!(jp)(i) = \Delta^{\{0, \ldots, i\}} \times_{\Delta^n} A.
\end{equation*}
For $i<n$, an application of Lemma \ref{lem:finalfiber} to the left fibration obtained by pulling $A$ back to the $i$-simplex $\Delta^{\{0, \ldots, i\}}$ shows that the first map is a weak homotopy equivalence.  To treat the case $i = n$, we should show that $A_n \rightarrow A$ is a weak homotopy equivalence. This follows from case (b) of Lemma \ref{lem:finalfiber2}.
\end{proof}

\section{A version of Quillen's Theorem A}
\label{sec:QuillenA}

As an illustration of the usefulness of our methods, we give a concise proof of Quillen's Theorem A for $\infty$-categories, which is formulated as Proposition \ref{prop:QuillenA} in the introduction. 

The case where $B$ is simply the nerve of a category $\A$ is almost immediate from what we have done so far. Indeed, let $f: X \rightarrow Y$ be a map of simplicial sets over $\NA$. The assumption is that for every object $a$ of $\A$, the map $f$ induces a weak equivalence of simplicial sets $X/a \rightarrow Y/a$. Note that $X/a = r_!(X)(a)$ and similarly for $Y$, so that this assumption translates into $r_!(f)$ being a weak equivalence of simplicial diagrams on $\A$. But $r_!$ is a left Quillen equivalence and hence detects weak equivalences between cofibrant objects, so that we immediately conclude that the original map $f: X \rightarrow Y$ is a weak equivalence in the covariant model structure over $\NA$. 

With a little more work, we will prove Proposition \ref{prop:QuillenA} over a general $\infty$-category $B$. Recall that for a map $\pi: X \rightarrow B$, we constructed a bisimplicial set $X // B$, for which an $(m,n)$-simplex in $(X//B)_{m,n}$ is a pair $(\xi \in X_m, \beta \in B_{m+n+1})$ such that $\pi(\xi) = \beta|_{\{0, \ldots, m\}}$. Restricting $\beta$ to $\{m+1, \ldots, m+n+1\}$ gives an augmentation of $X//B$ over the `constant'  bisimplicial set $cB$ defined by $cB_{m,n} = B_n$. Note that the category of bisimplicial sets augmented over $cB$ is equivalent to the diagram category
\begin{equation*}
\mathbf{sSets}^{(\mathbf{\Delta}/B)^{\mathrm{op}}},
\end{equation*}
where $\mathbf{\Delta}/B$ is the category of simplices of $B$. Under this identification, the value $(X//B)(\beta)$ on an $n$-simplex $\beta \in B_n$ is precisely the simplicial set $X/\beta$ discussed earlier. 

\begin{lemma}
\label{lem:finalface}
Let $B$ be an $\infty$-category (i.e. have the right lifting property with respect to all inner horn inclusions) and let $\beta$ be an $n$-simplex of $B$. Then the map 
\begin{equation*}
X/\beta \longrightarrow X/d_n \beta
\end{equation*}
is a trivial fibration of simplicial sets.
\end{lemma}
\begin{proof}
Consider a lifting problem
\[
\xymatrix{
\partial \Delta^m \ar[d] \ar[r]^\xi & X/\beta \ar[d] \\
\Delta^m \ar[r]^-\zeta \ar@{-->}[ur] & X/d_n\beta.
}
\]
For each $0 \leq i \leq m$, the map $\xi$ defines a simplex which we informally write as
\begin{equation*}
\pi(x_0) \rightarrow \cdots \rightarrow \widehat{\pi(x_i)} \rightarrow \cdots \rightarrow \pi(x_m) \rightarrow b_0 \rightarrow \cdots \rightarrow b_n,
\end{equation*}
where the hat denotes omission. The map $\zeta$ defines a simplex
\begin{equation*}
\pi(x_0) \rightarrow \cdots \rightarrow \pi(x_m) \rightarrow b_0 \rightarrow \cdots \rightarrow b_{n-1}.
\end{equation*}
Thus, our lifting problem is equivalent to a lifting problem
\[
\xymatrix{
\Lambda_S^{m+n+1} \ar[d]\ar[r] & B \\
\Delta^{m+n+1} \ar@{-->}[ur], & 
}
\]
where $S = \{m+1, \ldots, m+n\}$ and $\Lambda_S^{m+n+1}$ denotes the union of all the faces $d_i\Delta^{m+n+1}$ for $i$ \emph{not} contained in $S$. Since the missing faces $d_i\Delta^{m+n+1}$ for $i \in S$ are all inner faces, this horn inclusion is inner anodyne and a lift exists by the assumption that $B$ is an $\infty$-category. \end{proof}

\begin{corollary}
\label{cor:finalface}
Let $\beta$ be as in the previous lemma and let $b_0$ denote its initial vertex. Then $X/\beta \rightarrow X/b_0$ is a trivial fibration.
\end{corollary}

\begin{corollary}
\label{cor:equivfatslice}
Suppose $f: X \rightarrow Y$ is a map of simplicial sets over $B$ such that for every vertex $b$ of $B$, the induced map $X/b \rightarrow Y/b$ is a weak equivalence of simplicial sets in the Kan-Quillen model structure. Then the induced map $X // B \rightarrow Y // B$ is a pointwise weak equivalence of diagrams on $(\mathbf{\Delta}/B)^{\mathrm{op}}$.
\end{corollary}

Note that taking diagonals induces a functor from the category of bisimplicial sets over $cB$ to the category of simplicial sets over $B$.

\begin{proposition}
\label{prop:diag}
Suppose $g: V \rightarrow W$ is a map of bisimplicial sets over $cB$ such that for every simplex $\beta$ of $B$, the induced map $V(\beta) \rightarrow W(\beta)$ is a weak equivalence of simplicial sets. Then $\mathrm{diag}(V) \rightarrow \mathrm{diag}(W)$ is covariant weak equivalence of simplicial sets over $B$.
\end{proposition}
\begin{proof}
The category $\mathbf{\Delta}/B$ is Reedy, so that we may consider the Reedy model structure on the category $\mathbf{sSets}^{(\mathbf{\Delta}/B)^{\mathrm{op}}}$. Its generating cofibrations and trivial cofibrations are the families of maps
\begin{equation*}
\partial\Delta^m \,\hat{\times} \,\beta \cup \Delta^m \,\hat{\times} \,\partial\beta \longrightarrow \Delta^m \,\hat{\times}\,\beta \quad\quad \text{for } m\geq 0
\end{equation*}
and
\begin{equation*}
\Lambda_k^m \,\hat{\times}\, \beta \cup \Delta^m \,\hat{\times}\, \partial\beta \longrightarrow \Delta^m \,\hat{\times}\,\beta \quad\quad \text{for } m \geq 0, \, 0\leq k \leq m
\end{equation*}
respectively, where $\beta$ is an $n$-simplex of $B$ and $\hat{\times}$ denotes the external product defined by $(X \,\hat{\times}\, Y)_{m,n} = X_m \times Y_n$. We claim that the diagonal functor is left Quillen, from the Reedy model structure just described to the covariant model structure on $\mathbf{sSets}/B$. Indeed, the diagonal obviously preserves colimits. Furthermore, it transforms the external product into an actual product, so that the preservation of (trivial) cofibrations is a direct consequence of the fact that the covariant model structure is simplicial. Being a left Quillen functor, the diagonal preserves weak equivalences between cofibrant objects. But it is a classical fact that in the Reedy model structure on bisimplicial sets over $cB$ every object is cofibrant; in fact, the Reedy model structure coincides with the injective one in this case.
\end{proof}

\begin{proof}[Proof of Proposition \ref{prop:QuillenA}]
Let $f: X \rightarrow Y$ be a map satisfying the assumptions of Proposition \ref{prop:QuillenA}. Then by Corollary \ref{cor:equivfatslice} and Proposition \ref{prop:diag}, the map $\mathrm{diag}(X//B) \rightarrow \mathrm{diag}(Y//B)$ is a covariant weak equivalence of simplicial sets over $B$. By Proposition \ref{prop:zigzag2}, the horizontal maps in the following diagram are covariant weak equivalences over $B$:
\[
\xymatrix{
X \ar[d]_f \ar[r]^-\iota & L(X) \ar[d] & \mathrm{diag}(X//B) \ar[l]_-\gamma \ar[d] \\
Y \ar[r]^-\iota & L(Y) & \mathrm{diag}(Y//B) \ar[l]_-\gamma.
}
\]
We conclude that $X \rightarrow Y$ itself must also be a covariant equivalence over $B$.
\end{proof}

In the above proof we invoked Proposition \ref{prop:zigzag2}, for which we still need the following:

\begin{lemma}
\label{lem://leftQ}
Let $B$ be an $\infty$-category and $\beta$ an $n$-simplex of $B$ with initial vertex $b_0$. Then $\mathrm{diag}(b_0//B) \rightarrow \mathrm{diag}(\beta//B)$ is a covariant weak equivalence over $B$.
\end{lemma}
\begin{proof}
We will make use of a trisimplicial set $b_0 // \beta // B$, of which a $(p,q,r)$-simplex $\xi$ is a pair $(\xi_1, \xi_2)$ as follows:
\begin{itemize}
\item[(i)] $\xi_1$ is a map $\Delta^p \star \Delta^q \rightarrow \Delta^n$.
\item[(ii)] $\xi_1$ is a degenerate simplex whose restriction to $\Delta^p$ is constant with value $0$.
\item[(iii)] $\xi_2$ is a map  $\Delta^p \star \Delta^q \star \Delta^r \rightarrow B$, i.e. a $(p+q+r+2)$-simplex of B.
\item[(iv)] The map $\xi_2$ should be compatible with $\xi_1$, in the sense that $\beta \xi_1$ is the restriction of $\xi_2$ to $\Delta^p \star \Delta^q$.
\end{itemize}
If we use the heuristic notation $b_0 \rightarrow b_1 \rightarrow \cdots \rightarrow b_n$ for $\beta$, we could depict such a simplex $\xi$ as
\begin{equation*}
\underbrace{b_0 = \cdots = b_0}_{p} \rightarrow \underbrace{b_{\xi_1(p+1)} \rightarrow \cdots \rightarrow b_{\xi_1(p+q+1)}}_{q} \rightarrow \underbrace{c_0 \rightarrow \cdots \rightarrow c_r}_r,
\end{equation*}
where $c_i = \xi_2(p+q+2+i)$. We use the notation $\mathrm{con}_p(\beta//B)$ for the trisimplicial set whose $(p,q,r)$-simplices are $(\beta//B)_{q,r}$ and $\mathrm{diag}_{p,q}(b_0 // \beta // B)$ for the bisimplicial set whose $(p,r)$-simplices are $(b_0 // \beta // B)_{p,p,r}$. First observe that we may factor the map $b_0//B \rightarrow \beta//B$ by maps
\begin{equation*}
b_0 // B \xrightarrow{\varphi} \mathrm{diag}_{p,q}(b_0 // \beta // B) \xrightarrow{\psi} \beta // B.
\end{equation*}
Here $\varphi$ sends a $(p,r)$-simplex $\xi$ of $b_0 // B$, which we may depict as
\begin{equation*}
\underbrace{b_0 = \cdots = b_0}_p \rightarrow \underbrace{c_0 \rightarrow \cdots \rightarrow c_r}_r,
\end{equation*}
to the $(p,r)$-simplex with `degenerate middle'
\begin{equation*}
\underbrace{b_0 = \cdots = b_0}_p = \underbrace{b_0 = \cdots = b_0}_{p} \rightarrow  \underbrace{c_0 \rightarrow \cdots \rightarrow c_r}_r
\end{equation*}
of $\mathrm{diag}_{p,q}(b_0 // \beta // B)$, more formally defined as the degenerate simplex $s_p^{p+1}\xi$. The map $\psi$ is the $(p,q)$-diagonal of a map $\Psi: b_0 // \beta // B \rightarrow \mathrm{con}_{p}(\beta//B)$, defined by sending a simplex $\xi = (\xi_1, \xi_2)$ of $b_0 // \beta // B$ to the simplex $(d_0^{p+1}\xi_1, d_0^{p+1}\xi_2)$ of $\beta //B$; more informally, it sends
\begin{equation*}
\underbrace{b_0 = \cdots = b_0}_{p} \rightarrow \underbrace{b_{\xi_1(p+1)} \rightarrow \cdots \rightarrow b_{\xi_1(2p+1)}}_{q} \rightarrow \underbrace{c_0 \rightarrow \cdots \rightarrow c_r}_r
\end{equation*}
to
\begin{equation*}
\underbrace{b_{\xi_1(p+1)} \rightarrow \cdots \rightarrow b_{\xi_1(2p+1)}}_{q} \rightarrow \underbrace{c_0 \rightarrow \cdots \rightarrow c_r}_r.
\end{equation*}
Note that $\varphi$ admits an evident retraction $\rho$, simply applying the $(p+1)$-fold face operator $d_{p+1}^{p+1}$. In fact, an easy argument (like the one at the end of the proof of Proposition \ref{prop:zigzag}) shows that for fixed $r$, the maps of simplicial sets
\begin{equation*}
\varphi: (b_0 // B)_{\bullet, r} \rightarrow \mathrm{diag}_{p,q}(b_0 // \beta // B)_{\bullet, r} \quad \text{and} \quad \rho: \mathrm{diag}_{p,q}(b_0 // \beta // B)_{\bullet, r} \rightarrow (b_0 // B)_{\bullet, r}
\end{equation*}
are part of a deformation retract. In particular, it follows from Proposition \ref{prop:diag} that $\mathrm{diag}(\varphi)$ is a covariant weak equivalence over $B$.

It remains to show that $\mathrm{diag}(\psi)$ is a covariant weak equivalence over $B$. We will show that for fixed $q$ and $r$, the map $\Psi$ gives a weak equivalence of simplicial sets
\begin{equation*}
\Psi_{\bullet,q,r}: (b_0 // \beta // B)_{\bullet,q,r} \longrightarrow \mathrm{con}_p(\beta//B)_{\bullet,q,r}.
\end{equation*}
It follows that $\psi$, which equals $\mathrm{diag}_{p,q}(\Psi)$, is a weak equivalence of simplicial sets for fixed $r$, so that another application of Proposition \ref{prop:diag} proves that $\mathrm{diag}(\psi)$ is indeed a covariant weak equivalence of simplicial sets over $B$.

To prove our claim about $\Psi_{\bullet,q,r}$, first observe that by the same arguments used to prove Lemma \ref{lem:finalface} and Corollary \ref{cor:finalface} we may reduce to the case $r=0$. In this case $\Psi_{\bullet,q,r}$ is a trivial fibration: indeed, consider a lifting problem
\[
\xymatrix{
\partial\Delta^p \ar[d]\ar[r]^-{f} & (b_0 // \beta // B)_{\bullet,q,0} \ar[d] \\
\Delta^p \ar[r]_-g & \mathrm{con}_p(\beta//B)_{\bullet,q,0}.
}
\]
Let $S$ be the set $\{p+1, \ldots, p+q+1\}$, which defines a horn $\Lambda_S^{p+q+2}$ consisting of all the faces $d_i\Delta^{p+q+2}$ for $i$ not in $S$. Then we may form a corresponding diagram
\[
\xymatrix{
\Lambda_S^{p+q+2} \ar[r]^-F\ar[d] & B \\
\Delta^{p+q+2}.
}
\]
Indeed, on a face $d_i\Delta^{p+q+2}$ for $i= 0, \ldots, p$ we define $F$ to be the simplex of $B$ corresponding to the evaluation of $f$ on the face $d_i\Delta^p$. To define $F$ on the final face $d_{p+q+2}\Delta^{p+q+2}$, note that $g$ defines a simplex of $B$ which we may depict as
\begin{equation*}
\underbrace{b_{i_0} \rightarrow b_{i_1} \rightarrow \cdots \rightarrow b_{i_q}}_q \rightarrow c_0.
\end{equation*}
We let $F|_{d_{p+q+2}\Delta^{p+q+2}}$ be the simplex
\begin{equation*}
\underbrace{b_0 = \cdots = b_0}_{p} \rightarrow \underbrace{b_{i_0} \rightarrow \cdots \rightarrow b_{i_q}}_q
\end{equation*}
of $B$ defined by taking the map $\Delta^{p+q+1} \rightarrow \Delta^n$ sending the first $p$ vertices to $0$ and the $(p+1+j)$'th to $i_j$ and then composing with $\beta: \Delta^n \rightarrow B$. One verifies that the maps on the various faces defined in this way assemble to give $F$ as in the diagram. Since $B$ is assumed to be an $\infty$-category, there exists an extension $\widehat{F}: \Delta^{p+q+2} \rightarrow B$, which in turn defines a map $\Delta^p \rightarrow (b_0 // \beta // B)_{\bullet,q,0}$ solving the original lifting problem.
\end{proof}

\section{A version of Quillen's Theorem B}
\label{sec:QuillenB}

In this final section we prove Proposition \ref{prop:QuillenB}. First we rephrase the condition of its statement:

\begin{lemma}
\label{lem:locconstant}
Let $X \rightarrow B$ be a map of simplicial sets and suppose $B$ is an $\infty$-category. If the map $X/\beta \rightarrow X/d_0\beta$ is a weak equivalence of simplicial sets for every 1-simplex $\beta$ of $B$, then $X/\delta \rightarrow X/\gamma$ is a weak equivalence for any map $\gamma \rightarrow \delta$ in the category of simplices of $B$. In other words, the simplicial diagram $X//B$ on $(\mathbf{\Delta}/B)^{\mathrm{op}}$ is homotopy constant.
\end{lemma}
\begin{proof}
Consider the following two kinds of maps in $X//B$:
\begin{itemize}
\item[(i)] Maps $\gamma \rightarrow \delta$ sending the initial vertex of $\gamma$ to the initial vertex of $\delta$.
\item[(ii)] Face inclusions $d_0\beta \rightarrow \beta$, where $\beta$ is a simplex of dimension at least 1.
\end{itemize}
Any map in $\mathbf{\Delta}/B$ is a composition of maps of the types (i) and (ii), so it suffices to treat these. For a map of type (i), we may form the following square:
\[
\xymatrix{
X/\delta \ar[r]\ar[d] & X/\delta(0) \ar@{=}[d] \\
X/\gamma \ar[r] & X/\gamma(0).
}
\]
The horizontal maps are trivial fibrations by Corollary \ref{cor:finalface}, so that the left vertical map is a weak equivalence by two-out-of-three. For a map of type (ii), write $\beta'$ for the restriction $\beta|_{\Delta^{\{0,1\}}}$ of $\beta$ to its first two vertices. Consider the diagram
\[
\xymatrix{
X/\beta \ar[d]\ar[r] & X/\beta' \ar[d] \\
X/d_0\beta \ar[r] & X/d_0\beta'.
}
\]
Again, the horizontal maps are trivial fibrations and now the right vertical map is a weak equivalence by assumption. It follows that the left vertical map is a weak equivalence.
\end{proof}

To compute the homotopy fiber of $p: X \rightarrow B$, we may as well replace it by the map $\mathrm{diag}(X//B) \rightarrow B$. Indeed, by Proposition \ref{prop:zigzag2}, these maps have the same homotopy type in the covariant model structure over $B$. The Kan-Quillen model structure is a localization of the covariant one, so that the same is true there. Consider the well-known `initial vertex' map $\lambda: N(\mathbf{\Delta}/B)^{\mathrm{op}} \rightarrow B$. It sends a simplex $\beta$ of $B$ to its initial vertex $\beta(0)$; more generally, it sends a chain of simplices
\begin{equation*}
\beta_0 \xleftarrow{f_1} \beta_1 \xleftarrow{f_2} \cdots \xleftarrow{f_n} \beta_n
\end{equation*}
to the simplex
\begin{equation*}
\beta_0\bigl(0 \rightarrow f_1(0) \rightarrow f_2f_1(0) \rightarrow \cdots \rightarrow f_n \cdots f_1(0)\bigr).
\end{equation*}
It is a classical fact that $\lambda$ is a weak equivalence. To prove this, note that $N(\mathbf{\Delta}/-)^{\mathrm{op}}$, as a functor from the category of simplicial sets to itself, preserves colimits and monomorphisms. By skeletal induction one then reduces to the case of a simplex $\Delta^n$, where it is straightforward to see that $N(\mathbf{\Delta}/\Delta^n)^{\mathrm{op}}$ is weakly contractible. 

\begin{lemma}
\label{lem:kappa}
There is a weak homotopy equivalence $\kappa: h_!(X//B) \rightarrow \mathrm{diag}(X//B)$ making the following square commute:
\[
\xymatrix{
\mathrm{diag}(X//B) \ar[d] & h_!(X//B) \ar[l]_-{\kappa}\ar[d] \\
B & N(\mathbf{\Delta}/B)^{\mathrm{op}} \ar[l]^-\lambda.
}
\]
\end{lemma}
\begin{proof}
As in the previous section, we have interpreted $X//B$ as a simplicial diagram on $(\mathbf{\Delta}/B)^{\mathrm{op}}$ in order to apply $h_!$ to it. The map $\kappa$ will be the diagonal of a map of bisimplicial sets 
\begin{equation*}
K: N\Bigl(\int_{(\mathbf{\Delta}/B)^{\mathrm{op}}} X//B\Bigr) \longrightarrow X//B.
\end{equation*}
For fixed $m$, we have the following identifications of the simplicial sets involved:
\begin{eqnarray*}
(X//B)_{m, \bullet} & = & \coprod_{\xi_1 \in X_m} p(\xi_1)/B \\
N\Bigl(\int_{(\mathbf{\Delta}/B)^{\mathrm{op}}} X//B\Bigr)_{m, \bullet} & = & \coprod_{\xi_1 \in X_m} N\Bigl(\int_{\mathbf{\Delta}^{\mathrm{op}}}p(\xi_1)/B\Bigr).
\end{eqnarray*}
Note that we have written $\int_{\Delta^{op}} A$ rather than $\mathbf{\Delta}/A$ for the category of simplices of a simplicial set $A$, to avoid obvious notational issues. There is an evident map between the two simplicial sets above, namely the coproduct of initial vertex maps
\begin{equation*}
\lambda: N\Bigl(\int_{\mathbf{\Delta}^{\mathrm{op}}}p(\xi_1)/B\Bigr) \longrightarrow p(\xi_1)/B.
\end{equation*}
One checks that this assignment is also natural in $m$, giving the desired map $K$. We already observed that these initial vertex maps are weak homotopy equivalences, so that the conclusion of the lemma follows upon taking diagonals.
\end{proof}

\begin{proof}[Proof of Proposition \ref{prop:QuillenB}]
By virtue of Lemma \ref{lem:kappa} we have reduced our task to computing the homotopy fiber of the map
\begin{equation*}
h_!(X//B) \longrightarrow N(\mathbf{\Delta}/B)^{\mathrm{op}}
\end{equation*}
over a vertex $b$, regarded as an object of $(\mathbf{\Delta}/B)^{\mathrm{op}}$. Write $X//B \rightarrow (X//B)_f$ for a fibrant replacement of $X//B$ in the projective model structure for simplicial diagrams on $(\mathbf{\Delta}/B)^{\mathrm{op}}$. In other words, this map is a weak equivalence of simplicial sets when evaluated on any simplex $\beta$ of $B$ and $(X//B)_f$ takes values in Kan complexes. Then the adjoint of the map of Proposition \ref{prop:proof1} provides the first of the following two maps:
\begin{equation*}
h_!(X//B) \longrightarrow r^*(X//B) \longrightarrow r^*\bigl((X//B)_f\bigr).
\end{equation*}
The composition of these two maps is a covariant weak equivalence (so in particular a weak homotopy equivalence) by Theorem \ref{thm:c}. By Lemma \ref{lem:locconstant}, the diagram $(X//B)_f$ is a fibrant object in the localized model structure of Proposition \ref{prop:localized} on the category of simplicial diagrams on $N(\mathbf{\Delta}/B)^{\mathrm{op}}$. Therefore $r^*\bigl((X//B)_f\bigr) \rightarrow N(\mathbf{\Delta}/B)^{\mathrm{op}}$ is a Kan fibration. The fiber of this map over $b$ is precisely the value $(X//B)_f(b)$, which is by assumption a Kan complex weakly equivalent to the simplicial set $(X//B)(b) = X/b$, completing the proof.
\end{proof}

\bibliographystyle{plain}
\bibliography{biblio}

\end{document}